\documentclass[final,3p,times]{elsarticle}

\usepackage[T1]{fontenc}
\usepackage{tabularx,ragged2e,booktabs,caption}
\newcolumntype{C}[1]{>{\Centering}m{#1}}

 \usepackage{graphics}
\def\RR{{\mathbb R}} 
\usepackage{array} 
\newcolumntype{M}[1]{>{\centering\arraybackslash}m{#1}} 
\newcolumntype{N}{@{}m{0pt}@{}}

\newcommand{\be} {\begin{equation}}
\newcommand{\ee} {\end{equation}}

\newcommand{\A} {\mathcal{A}}

\usepackage[latin1]{inputenc}
\usepackage{amsmath,amsfonts,eufrak,amssymb,graphics}
\usepackage[mathcal]{eucal}
\numberwithin{equation}{section}
\usepackage{amsthm}
\usepackage{fancyhdr}
\usepackage{eucal}

\newtheorem{theorem}{Theorem}[section]
\newtheorem{lemma}[theorem]{Lemma}
\newtheorem{proposition}[theorem]{Proposition}
\newtheorem{corollary}[theorem]{Corollary}
\newtheorem{e-definition}{Definition\rm}
\newtheorem{remark}{Remark\/}[section]
\newtheorem{example}{Example\/}

\newproof{pf}{Proof}
\newproof{pott}{Proof of Theorem \ref{pdq}}
\def\neweq#1{\begin{equation}\label{#1}}
\def\endeq{\end{equation}}

\def\phi{\varphi}
\def\RR{{\mathbb R} }

\usepackage[notref]{showkeys}

\usepackage{amssymb}
 \usepackage{amsthm}
 \usepackage{fancyhdr}
\usepackage{eucal}
\usepackage[usenames, dvipsnames]{color}
\usepackage[perpage]{footmisc}
\usepackage{datetime}

\makeatletter
\def\ps@pprintTitle{
 \def\@oddfoot{\footnotesize {\em \today}\hspace{15cm}}
 }
\makeatother

\begin{document}
\begin{frontmatter}

\title{Singular solutions for divergence-form elliptic equations involving regular variation theory: Existence and classification\tnoteref{title}}
\tnotetext[title]{The second author was supported by The Australian Research Council
Grant No. DP120102878 ``Analysis of non-linear partial differential equations describing singular phenomena''.}
\author[T.-Y. Chang]{Ting-Ying Chang}
\ead{T.Chang@maths.usyd.edu.au}
\author[F.C. Cirstea]{Florica C. C\^irstea\corref{cor1}}
\ead{florica.cirstea@sydney.edu.au}
\cortext[cor1]{Corresponding author.}
\address{School of Mathematics and Statistics, The University of Sydney,
 NSW 2006, Australia}


\begin{abstract} 

We generalise and sharpen several recent results in the literature regarding the existence and  
complete classification of the isolated singularities 
for a broad class of nonlinear elliptic equations of the form
\neweq{for} -{\rm div}\,(\mathcal A(|x|) \,|\nabla u|^{p-2} \nabla u)+b(x)\,h(u)=0\quad \text{in } B_1\setminus\{0\},
\endeq
where $B_r$ denotes the open ball with radius $r>0$ centred at $0$ in $\RR^N$ $(N\geq 2)$.  
We assume that $\A \in C^1(0,1]$, $b\in C(\overline{B_1}\setminus\{0\})$ and 
$h\in C[0,\infty)$ are positive functions associated with regularly varying 
functions of index $\vartheta$, $\sigma$ and $q$ at $0$, $0$ and $\infty$ 
respectively,  satisfying $q>p-1>0$ and $\vartheta-\sigma<p<N+\vartheta$. 
We prove that the condition $b(x) \,h(\Phi)\not \in L^1(B_{1/2})$ is sharp for the removability of all singularities at $0$ for the positive solutions of \eqref{for}, where $\Phi$ denotes the ``fundamental solution'' of  
$-{\rm div}\,(\mathcal A(|x|)\, |\nabla u|^{p-2} \nabla u)=\delta_0$ (the Dirac mass at $0$) in $B_1$, subject to $\Phi|_{\partial B_1}=0$. 
If $b(x) \,h(\Phi)\in L^1(B_{1/2})$, we show that any non-removable singularity at $0$ for 
a positive solution of \eqref{for} is either {\em weak} (i.e., 
$\lim_{|x|\to 0} u(x)/\Phi(|x|)\in (0,\infty)$) or {\em strong} ($ \lim_{|x|\to 0} u(x)/\Phi(|x|)=\infty$).  
The main difficulty and novelty of this paper, 
for which we develop new techniques, come from the explicit 
asymptotic behaviour of the strong singularity solutions in the critical case, which had 
previously remained open even for $\A=1$. We also study the existence and uniqueness of the positive 
solution of \eqref{for} with a prescribed
admissible behaviour at $0$ and a Dirichlet condition on $\partial B_1$.  

\end{abstract}


\end{frontmatter}

\tableofcontents

\section{Introduction and main results} \label{introd}
The local behaviour of solutions for nonlinear partial differential equations of second order 
has been studied extensively in the last fifty years (see, for example, 
V\'eron \cite{veron96} for relevant background). The topic of isolated singularities represents an  
extremely active area of research concerning many different classes of nonlinear elliptic equations.  
Recent contributions 
include, on the one hand, boundary singularities (see, for example, \cite{mn,phve})
and, on the other hand, 
interior singularities for the fractional Laplacian \cite{caf,chenv}, the weighted $p$-Laplacian \cite{YS},   
non-homogeneous operators in divergence form \cite{mihai}, nonlinear equations with singular potentials \cite{florica,fp2} or with nonlinearities depending on the gradient \cite{mml,chci} to name only a few.     

Motivated by previous articles such as \cite{BCCT,florica,cd,fav,vv}, 
we aim to obtain a complete understanding of the isolated 
singularities for nonlinear elliptic equations of the form \eqref{une1} in the punctured unit ball $B_1\setminus\{0\}$ in $\RR^N$ ($N\geq 2$)
under the Assumptions $({\mathbf A_1})$--$({\mathbf A_3})$ given later. 
A prototype model is $\A(|x|)=|x|^\vartheta$, $b(x)=|x|^\sigma$ and $h(t)=|t|^{q-1}t$ for $q>p-1>0$ and $\vartheta-\sigma<p\leq N+\vartheta$. 
In the standard case $\A=b=1$, the profile of all positive solutions of 
$p$-Laplacian type equations with pure power nonlinearities, namely $ {\rm div} \, (|\nabla u|^{p-2}\nabla u)=|u|^{q-1} u$ in $B_1\setminus\{0\}$, is well clarified  (see \cite{fav,vv}), 
depending on the position of $q$ relative to the {\em critical exponent} 
$q_*=\frac{N(p-1)}{N-p}$ (with $q_*=\infty$ for $p=N$):
\begin{enumerate}
\item[{\rm (a)}] If $p-1<q<q_*$, then as $|x|\to 0$, exactly one of the following holds (see Friedman--V\'eron \cite{fav}): 
\begin{enumerate}
\item[{\rm (i)}] $u$ can be extended as a continuous solution of the same equation in $B_1$ (removable singularity);  
\item[{\rm (ii)}] There exists a positive number $\lambda$ such that $u(x)/\mu(x)\to \lambda $ (weak singularity) and, moreover,
$$ - {\rm div} \, (|\nabla u|^{p-2}\nabla u)+|u|^{q-1}u=\lambda^{p-1} \delta_0\quad \text{in }\ {\mathcal D}'(B_1). $$
Here, $\delta_0$ denotes the Dirac mass at $0$, whereas $\mu$ stands for the fundamental solution of the $p$-harmonic equation
$-{\rm div}\,(|\nabla u|^{p-2}\nabla u)=\delta_0$ in 
 $\mathcal D'(\RR^N)$ (in the sense of distributions). 
\item[{\rm (iii)}] $|x|^{p/(q-p+1)}u(x)\to \gamma_{N,p,q}$, where $\gamma_{N,p,q}:=\left[\left(\frac{p}{q-p+1}\right)^{p-1}
\left(\frac{pq}{q-p+1}-N\right)\right]^{1/(q-p+1)}$ (strong singularity).  
\end{enumerate}
\item[{\rm (b)}] If, in turn, $q\geq q_*$ (for $1<p<N$), then only (a)(i) occurs, 
(see V\'azquez--V\'eron \cite{vv}).
\end{enumerate}

The alternatives (i)--(iii) in (a) correspond respectively to a positive solution $u$
with $ \limsup_{|x|\to 0} u(x)/\mu(x)$ equal to zero, a positive finite number, 
and infinity. Furthermore, if $g\in C^1(\partial B_1)$ is a non-negative function and 
$\lambda\in (0,\infty)$, then the singular Dirichlet problem
 $ {\rm div} \, (|\nabla u|^{p-2}\nabla u)=|u|^{q-1} u$ in $B_1\setminus\{0\}$, 
with $\lim_{|x|\to 0} u(x)/\mu(x)=\lambda$ and $u=g$ on $\partial B_1$ admits a 
unique non-negative solution if and only if $q<q_*$ (see \cite[Theorems 1.1 and 1.2]{fav}). 
The positive solutions with a strong singularity at $0$ are {\em all} obtained as limits 
of solutions with a weak singularity at $0$. However, going beyond the power nonlinearities, 
the understanding of strong singularities had until now remained elusive. 
The removability of the strong singularity solutions is not completely clear even for 
Laplacian-type equations. The following question formulated by V\'azquez and
V\'eron \cite{vv2} is still open: {\em What is the weakest condition on a continuous 
non-decreasing function $h$ such that any isolated singularity of a non-negative 
solution of $\Delta u=h(u) $ in $B_1\setminus\{0\}$ with $N\geq 3$ is removable?}
By \cite[Remark~2.2]{vv2}, there are examples of continuous non-decreasing functions $h$ satisfying
\neweq{remv} 
\int_1^\infty  t^{-\frac{2(N-1)}{N-2}}\,h(t)\,dt=\infty\quad \text{and}\quad \int_1^\infty \frac{dt}{\sqrt{th(t)}}=\infty
\endeq
for which there exist no positive solutions with a weak singularity at $0$, 
but infinitely many positive solutions with a strong singularity at $0$. 
It is known (see \cite{vv2} or \cite{ve2}) that a necessary and sufficient 
condition for the removability of the weak singularities of 
the positive solutions is that $h$ satisfies the first integral condition in \eqref{remv}. 
Recently, the above question, together with a complete classification of the isolated singularities, 
has been settled by C\^irstea \cite{florica} in the framework of regular variation theory for more general semilinear elliptic equations.  

In two pioneering works, Serrin \cite{serrin64,serrin65} studied {\em a priori} estimates of solutions, the 
nature of removable singularities, and the behaviour of a positive solution 
in the neighbourhood of an isolated singularity for quasi-linear elliptic 
equations of the general form
\neweq{gform} 
{\rm div} \,{\mathbf A}(x,u, \nabla u)=B(x,u,\nabla u). 
\endeq 
For a domain $\Omega$ in $\RR^N$ with $0\in \Omega$, 
it is assumed that ${\mathbf A}(x,u,\xi)$ and $B(x,u,\xi)$ are, respectively, 
vector and scalar measurable functions defined in $\Omega\times \RR\times \RR^N$ 
satisfying the following growth conditions:
\neweq{gr} 
\left\{ 
\begin{aligned}
& |{\mathbf A}(x,u,\xi) |\leq \beta_0 |\xi|^{p-1}+\beta_1|u|^{p-1}+\beta_2,\\ 
& \xi\cdot {\mathbf A}(x,u,\xi)\geq  |\xi|^p-\beta_3 |u|^p-\beta_4,\\
&  |B(x,u,\xi)| \leq \beta_6 |\xi|^{p-1}+\beta_3|u|^{p-1}+\beta_5 \quad \text{for all } (x,u,\xi)\in \Omega\times \RR\times \RR^N,
\end{aligned} 
\right.
\endeq
where $1<p\leq N$ is a fixed exponent, $\beta_0$ is a positive constant and 
$\beta_i$ ($1\leq i\leq 6$) are measurable functions on $\Omega$ belonging 
to suitable Lebesgue classes: $\beta_1,\beta_2\in L^{N/(p-1-\varepsilon)}$, 
$\beta_6\in L^{N/(1-\varepsilon)}$ and 
$\beta_j\in L^{N/(p-\varepsilon)}$ for $j=3,4,5$, where $\varepsilon>0$. 
By \cite[Theorem~1]{serrin65}, 
if $u$ is a non-negative continuous solution of 
\eqref{gform} in $\Omega\setminus\{0\}$, then the following dichotomy holds:
\begin{enumerate}
\item either 
$u$ has a removable singularity at $0$;
\item or 
there exist positive constants $c_1$ and $c_2$ such that $c_1\leq u(x)/ 
\mu(|x|)\leq c_2 $ in a neighbourhood of zero.  
\end{enumerate}

In this paper, we address the singularity problem for quasi-linear elliptic 
equations in divergence form related to \eqref{gform} when {\em the growth
of $B$ is bigger than that of ${\mathbf A}$}, which is a challenge formulated by V\'eron \cite{veron96}. 
In this case, the main difficulty lies in the fact that solutions with strong singularities may appear.

The main feature of our study is to reveal a sharp and complete classification of the isolated singularities of 
\begin{equation} \label{une1}
{\rm div}\,(\mathcal A(|x|)\, |\nabla u|^{p-2} \nabla u) =b(x) \,h(u)\quad \text{in } B^*:=B_1\setminus\{0\}
\end{equation}
for a large class of nonlinearities, including model cases departing from power functions such as in Table~\ref{tab:title} below. 

\bigskip
\begin{minipage}{\linewidth}
\centering
\begin{tabular}{|M{1.5cm}|M{5cm}|M{3cm}|M{2.5cm}|N}
\hline
\bf   Example & $\A(|x|)$ as $|x|\to 0$ & $b(x)$ as $|x|\to 0$ & $h(t)$ as $t\to \infty$ &\\[10pt]
\hline
    $1$	&  $\displaystyle |x|^\vartheta \left(\ln \frac{1}{|x|}\right)^\alpha$ &  $\displaystyle|x|^\sigma\left(\ln \frac{1}{|x|}\right)^\beta$ 	& 
    $\displaystyle t^q(\ln t)^\gamma$ &\\[20pt]
\hline
    $2$	& $\displaystyle |x|^\vartheta \left(\ln \frac{1}{|x|}\right)^\alpha$ &  $\displaystyle |x|^\sigma\left(\ln \frac{1}{|x|}\right)^\beta$ & 
    $\displaystyle t^q\exp\left\{-(\ln t)^\nu\right\}$ &\\[20pt]
\hline 
 $3$	& $\displaystyle |x|^\vartheta \left(\ln \frac{1}{|x|}\right)^\alpha \exp\left\{-\frac{p-1}{q} \sqrt{\ln \frac{1}{|x|}}\right\}$ &  $\displaystyle |x|^\sigma\exp\left\{-\sqrt{\ln \frac{1}{|x|}}\right\}$ & $\displaystyle t^q\exp\left\{-(\ln t)^\nu\right\}$	&\\[20pt]
\hline
\end{tabular}\par
\captionof{table}{Examples} \label{tab:title} 
\medskip
\end{minipage}

In Examples~1--3 above, we take 
$\alpha$, $\beta$, $\gamma\in \RR$ and $\nu\in (0,1/2)$, $1<p< N+\vartheta$ and $q+1>p>\vartheta-\sigma$ (see Corollary~\ref{mil} for the classification). More generally, we work in the setting of regular variation theory inspired by 
C\^irstea and Du \cite{cd}, whose results (with $\A=1$) 
are here generalised and sharpened. 
The introduction of the weight function $\A$ in the operator in \eqref{une1} adds non-trivial difficulties. The first two conditions in \eqref{gr} are no longer
satisfied for $\mathbf A(x,u,\xi)=\A(|x|)\, |\xi|^{p-2}\xi$
since if $\vartheta>0$ (resp., $\vartheta<0$), then $\lim_{r\to 0^+} \A(r)=0$ (resp., $\infty$). A clear influence of $\A$ is felt in  
the behaviour of a positive solution of \eqref{une1} being compared not with $\mu$ but with a suitable ``fundamental solution'' $\Phi$ 
of the divergence-form 
operator $ {\rm div}\,(\mathcal A(|x|)\, |\nabla (\cdot )|^{p-2} \nabla (\cdot ))$ in $B_1$, see \eqref{2.1}. 

\vspace{0.2cm}
We assume the following structural conditions:

\begin{enumerate}
\item[{($\mathbf A_1$)}]
The function $\A\in C^1(0,1]$ is positive such 
that $\A(t)=t^\vartheta L_\A(t)$ with $1<p< N+\vartheta$ 
and $L_\A$ satisfies 
\begin{equation} \label{a3} 
 \lim_{t\to 0^+} \frac{t L_\A'(t)}{L_{\A}(t)}=0.
\end{equation}  

\item[{($\mathbf A_2$)}]
The function $h$ is continuous on $\RR$ and 
positive on $(0,\infty)$ with $h(0)=0$ and $h(t)/t^{p-1}$ bounded for small 
$t>0$, whereas $b$ is a positive continuous function on $\overline{B_1}\setminus\{0\}$.

\item[{($\mathbf A_3$)}]
There exist $q,\sigma\in \RR$ and functions $L_h$, $L_b$ 
that are slowly varying at $\infty$ and at $0$ respectively, such that
\begin{equation} \label{aoz}
\lim_{t\to \infty} \frac{h(t)}{t^q L_h(t)}=1 \ \text{and} \ 
\lim_{|x|\to 0} \frac{b(x)}{|x|^\sigma L_b(|x|)}=1 \  \text{ with } q+1>p>\vartheta-\sigma.
\end{equation}
\end{enumerate}

The condition in \eqref{a3} implies 
that $L_\A$ is slowly varying at $0$ (see Definition~\ref{def-reg} and Remark~\ref{nsw} in Appendix~\ref{appe}). 
A complete characterisation of slowly varying function at $0$ is provided by Theorem~\ref{rep-th}.
Without loss of generality, we assume that $L_h$ and $L_b$ satisfy the properties in \eqref{nor} (see Remark~\ref{asb} in Appendix~\ref{appe}).
For instances of $L_h$, one could choose any of the slowly varying functions at $\infty$ gathered in  
Example~\ref{example1} of Appendix~\ref{appe}. 
We note that the results of this paper can be extended for the case $p=N+\vartheta$ for certain cases of $L_\A$, see Remark~\ref{equal}.

\begin{e-definition} \label{defs} 
A function $u$ is said to be a {\em solution} ({\em sub-solution,
super-solution}) of \eqref{une1} if $u(x)\in C^1(B^*)$ and  
\begin{equation} \label{umn} \int_{B_1} \mathcal A(|x|)\, |\nabla u|^{p-2} \nabla u\cdot \nabla \phi\,dx +\int_{B_1} 
b(x)\,h(u)\,\phi\,dx=0   \qquad (\leq 0,\ \geq 0)\end{equation}
for all
functions (non-negative functions) $\phi(x)$ in $C^1_c(B^*)$, the space of all $C^1(B^*)$-functions with compact support in $B^*$.  
Furthermore, a positive solution $u$ of \eqref{une1} is said to be 
{\rm extended as a positive continuous solution} of \eqref{une1} in $B_1$ 
if there exists $\lim_{|x|\to 0}u(x)\in (0,\infty)$, the function $\A(|x|)\, |\nabla u|^{p-1}$ belongs to $L^1_{\rm loc}(B_1)$ and 
\eqref{umn} holds for every 
$\phi\in C^1_c(B_1)$.
\end{e-definition}

If $u$ is a positive solution of \eqref{une1} with $\limsup_{|x|\to 0} u(x)<\infty$, then both integrals in \eqref{umn} are well-defined for every
$\phi\in C^1_c(B_1)$. 
Indeed, $b\in L^1_{\rm loc}(B_1)$ since $\sigma>-N$ (from ($\mathbf A_1$) and ($\mathbf A_3$)), whereas the gradient estimates in Lemma~\ref{reg} give that $\A(|x|)\, |\nabla u|^{p-1}\in L^1_{\rm loc}(B_1)$ since $t^{1-p}\A(t)$ is regularly varying at $0$ with index $\vartheta-p+1$ (greater than $-N$).

\vspace{0.2cm}
Throughout this paper, we are concerned with non-negative solutions of \eqref{une1}. By the strong maximum principle, any non-negative solution
of \eqref{une1} is either identically zero or  
positive in $B^*$. Indeed, the conditions in Theorem~2.5.1 of \cite{PS} are satisfied
on any subset $\Omega\subset \subset B_1\setminus\{0\}$ with $\tilde {\mathbf A}(x,u,\nabla u)=\A(|x|)\,|\nabla u|^{p-2}\nabla u$ and $\tilde B(x,u,\nabla u)=-b(x)\,h(u)$ 
since $\A\in C(0,1]$ is a positive function, while $h$ and $b$ satisfy the properties in Assumption~ $({\mathbf A_2})$.    

\vspace{0.2cm}
\textbf{Fundamental Solution $\Phi$.}
Let $C_{N,p}:=\left(N\omega_N\right)^{-1/(p-1)} $, where $\omega_N$ denotes the 
volume of the unit ball in $\RR^N$. Assuming $({\mathbf A_1})$, we can define 
the ``fundamental solution'' of the operator ${\rm div}\,(\mathcal A(|x|)\, |\nabla (\cdot )|^{p-2} \nabla (\cdot ))$ in $\mathcal D'(B_1)$, namely 
\begin{equation}
\label{2.1}
\Phi(r):= C_{N,p} \int_r^1 \left( \frac{t^{1-N-\vartheta}}{L_\A(t)}\right)^{\frac{1}{p-1}}\,dt\quad \text{for all } r\in (0,1]. 
\end{equation} 
Note that $- {\rm div}\,(\mathcal A(|x|)\, |\nabla \Phi|^{p-2} \nabla \Phi)=\delta_0$ in $\mathcal D'(B_1)$ and 
$\Phi=0$ on $\partial B_1$. Moreover, $\lim_{r\to 0^+} \Phi(r)=\infty$ since $1<p<N+\vartheta$. 
We note that both $r \mapsto \Phi(r)$ and $r\mapsto -r \,\Phi'(r)$ are 
regularly varying at $0^+$ of index $-m_2$, where $m_2$ is defined in \eqref{stri}. Under Assumption $({\mathbf A_1})$, using Karamata's Theorem (see Theorem~\ref{kar2}), we find that
\begin{equation} \label{kphi}
\lim_{r\to 0^+} \frac{\ln \Phi(r)}{\ln \left(1/r\right)}=\lim_{r\to 0^+}\Upsilon(r)=m_2,\quad \text{where } 
\Upsilon(r):=\frac{r \left|\Phi'(r)\right|}{\Phi(r)}=\frac{C_{N,p} r^{-m_0} \left[L_\A(r)\right]^{-\frac{1}{p-1}}}{\Phi(r)}.
\end{equation}

We provide in Theorem~\ref{th1.2}(b) the sharp criteria, namely $b(x) \,h(\Phi)\not\in L^1(B_{1/2})$, for the 
removability of all singularities of the positive solutions of \eqref{une1}. 
In the case of non-removable singularities, that is $b(x) \,h(\Phi)\in L^1(B_{1/2})$, we give 
a complete classification of the singularities of \eqref{une1} in 
Theorem~\ref{th1.2}(a), accompanied by corresponding existence results 
in Theorem~\ref{wsol}.  
Our analysis brings new understanding of the behaviour of the 
solutions to \eqref{une1} with strong singularities at zero as the 
perturbation technique introduced in \cite{cd} for the subcritical 
case is not applicable in the critical case. The main innovation we 
develop is a perturbation technique which enables us to 
give precise explicit asymptotic formulas for the  
behaviour of the strong singular solutions. 
Our Theorems~\ref{th1.2} and \ref{wsol} extend the corresponding optimal results 
in \cite{BCCT} where $p=2$, $b=1$ and 
$h(t)=|t|^{q-1}t$ in \eqref{une1}. While the understanding of strong singularity 
solutions for Laplacian-type equations with power-like non-linearities in \cite{BCCT}
relied on the earlier work of Taliaferro \cite{tal}, this is no longer 
possible in our general context of quasi-linear equations such as \eqref{une1}. 

From Assumptions $({\mathbf A_1})$--$({\mathbf A_3})$, it follows that $m_0$, $m_1$ and $m_2$ are all {\em positive}, where we define
\neweq{stri}
m_0 := \frac{p+\sigma-\vartheta}{q-p+1},\qquad \quad m_1:= \frac{q-p+1}{p-1},\qquad \quad
m_2:=\frac{N+\vartheta-p}{p-1}. 
\endeq
Let us now define $q_*$, which shall be henceforth referred to as a {\em critical exponent}, namely
\neweq{critex}
q_*=\frac{N+\sigma}{m_2}.
\endeq
\begin{remark} \label{remark2}
Assuming $({\mathbf A_1})$--$({\mathbf A_3})$, 
we note that $b(x) \,h(\Phi)\in L^1(B_{1/2})$ is equivalent to
\begin{equation} \label{phicondition}
\int_{0^+} r^{N-1+\sigma}L_b(r)\,h(\Phi (r))\, \mathrm{d} r < \infty.
\end{equation}
The integrand in \eqref{phicondition} varies regularly at $0$ with index 
$N-1+\sigma-m_2q$. Hence, if $q\not=q_*$ then \eqref{phicondition} holds if and only if $q<q_*$, where $q_*$
is given by \eqref{critex}. If $q=q_*$, then \eqref{phicondition} may hold in some cases and fail in others. 
For example, if $L_\A=L_b=1$ and $h(t)=t^{q_*}(\ln t)^\alpha$ for $t>0$ large, then \eqref{phicondition} holds if and only if $\alpha<-1$. 
\end{remark}

Later in Corollary~\ref{mil}, we illustrate Theorem~\ref{th1.2} on the Examples of Table~\ref{tab:title}. 
Note that $L_h$ in Example~1 satisfies 
\neweq{doii}
t\longmapsto L_h(e^t) \text{ is regularly varying at } \infty \text{ with index } \gamma\in \RR. 
\endeq
In Example~3, we see that $L_\A$ and $L_b$ satisfy the following property
\neweq{doi} 
t\longmapsto \left[L_{\mathcal A}(e^{-t})\right]^{-\frac{q}{p-1}} L_b(e^{-t}) \text{ is regularly varying at } \infty\ \text{with index } j \in \RR. 
\endeq

From a practical viewpoint, we thus need to check $b(x)\, h(\Phi)\in L^1(B_{1/2})$ only for $q=q_*$. 
In such a critical case, assuming either \eqref{doii} or \eqref{doi}, then  
$b(x)\, h(\Phi)\in L^1(B_{1/2})$ if and only if $F(r)<\infty$, where we define 
\neweq{fodef}
F(r):= \int_0^r \xi^{-1} \left[L_\A(\xi)\right]^{-\frac{q_*}{p-1}} L_b(\xi)\,L_h(1/\xi)\,d\xi\quad \text{for } r>0\ \text{small}.
\endeq
The function $F$ plays an important role in the asymptotic behaviour at zero for a strong singularity solution of \eqref{une1}. 

\vspace{0.2cm}
We now state our first main result. 

\begin{theorem}[Classification of singularities and sharp removability results] \label{th1.2}
Let Assumptions $({\mathbf A_1})$--$({\mathbf A_3})$ hold. 
\begin{enumerate}
\item[{\rm (a)}]
If $b(x)\, h(\Phi)\in L^1(B_{1/2})$, then for every positive solution $u$ of \eqref{une1}, exactly one of the following cases occurs:
\begin{enumerate}
\item[{\rm (i)}] $u$ can be extended as a positive continuous solution of \eqref{une1} in the whole ball $B_1$.
\item[{\rm (ii)}] $u$ has a weak singularity at $0$, that is $\lim_{|x|\to 0}u(x)/\Phi(x)= \lambda\in (0,\infty)$ and, moreover, 
$u$ verifies
\neweq{b11}
 -  {\rm div}\,(\mathcal A(|x|)\, |\nabla u|^{p-2} \nabla u) +b(x)\,h(u) = \lambda^{p-1} \delta_0 \quad \text{in } \mathcal{D}'(B_1).
 \endeq

\item[{\rm (iii)}] $u$ has a strong singularity at $0$. Moreover, 
$\lim_{|x|\to 0} u(x)/ \tilde u(|x|)=1$, where $ \tilde u$ is given by 
\begin{equation} \label{stro}
\int_{\tilde u(r)}^\infty \frac{t^{-\frac{q+1}{p}}}{[L_h(t)]^{\frac{1}{p}}}\,dt=\int_0^r \left[M
\frac{ \xi^{\sigma-\vartheta } L_b(\xi)}{L_\A(\xi)}\right]^{\frac{1}{p}} \, d\xi\quad \text{with } \frac{1}{M}:=
q-\frac{N+\sigma}{m_0} \quad \text{ if } q<q_*. 
\end{equation}
On the other hand, in the critical case
$q=q_*$, then $\lim_{|x|\to 0} u(x)/ \tilde u(|x|)=1$ for
$\tilde u$ given by 
\begin{equation} 
\left\{\begin{aligned}
& \tilde u(r)=\left[m_1 m_0^{\gamma+1-p}F(r)\right]^{-\frac{1}{q_*-p+1}} L_\A^{-\frac{1}{p-1}}(r)  \,r^{-m_0}
&&  \text{if } \eqref{doii}\ \text{holds},&\\
& \int_c^{\tilde u(r)} \left[F(1/t)
\right]^{\frac{1}{q_*-p+1}} dt= \left( m_1m_0^{-p-j} \right)^{-\frac{1}{q_*-p+1}} 
  L_\A^{-\frac{1}{p-1}}(r) \, r^{-m_0}
&&  \text{if } \eqref{doi}\ \text{holds},&
\end{aligned}\right. \label{fg1}
\end{equation}
where $m_0,m_1$ and $F$ are prescribed by \eqref{stri} and \eqref{fodef}, respectively.  
In \eqref{fg1}, 
$c>0$ is a large constant. \end{enumerate}
\item[{\rm (b)}] If $b(x) \,h(\Phi)\not\in L^1(B_{1/2})$, then $q\geq q_*$ and
every positive solution of \eqref{une1} satisfies (a)(i).
\end{enumerate}
\end{theorem}

\begin{remark} \label{history} 
\begin{enumerate}
\item[{\rm (i)}]
 When $\A=b=1$ and $h(t)=|t|^{q-1} t$, our Theorem~\ref{th1.2}(a) recovers \cite[Theorem~2.1]{fav}. 
Moreover, Theorem~\ref{th1.2}(a) generalises and sharpens \cite[Theorem~1.1]{cd}, 
which analysed the case $\A=1$ and $q<q_*$. Our Theorem~\ref{th1.2} is also 
established under the optimal condition for the existence of solutions with 
singularities at $0$ for \eqref{une1}. Even for $\A=1$, the behaviour of the 
strong singularity solutions in the critical case $q=q_*$ is new, being 
obtained via a perturbation technique we devise in this paper (see \S\ref{sscrit}).

\item[{\rm (ii)}] When $\A=b= 1$ and $h(t)=t^{q}$, Theorem~\ref{th1.2}(b) recovers the removability result of 
\cite{brve} for $p=2$ and \cite{vv} for $1<p<N$. By letting $\A=1$ in Theorem~\ref{th1.2}(b), we also obtain a   
sharp version of \cite[Theorem~1.3]{cd}.  
\end{enumerate}
\end{remark}

{\bf Notation.} 
By $f_1(t)\sim f_2(t)$ as
$t\to t_0$ for $t_0\in \RR\cup\{\infty\}$, we mean that $\lim_{t\to t_0} f_1(t)/f_2(t)=1$. 

\vspace{0.2cm}
In Theorem~\ref{th1.2} for $q<q_*$, the function $\tilde u$ in \eqref{stro} is well-defined, regularly varying at $0$ with index $-m_0$ and 
\neweq{sqa}
\tilde u(r) \left[ L_h(\tilde u(r))\right]^{\frac{1}{q-p+1}}\sim \left[ \frac{m_0^p}{M} \frac{L_\A(r)}{L_b(r)}\right]^{\frac{1}{q-p+1}}r^{-m_0}\quad
\text{as } r\to 0^+. 
\endeq
Indeed, the integral in the left-hand side of \eqref{stro} is well-defined 
since the integrand is regularly varying at $\infty$ with index 
$-(q+1)/p <-1$ from the assumption $q>p-1$. The right-hand side of 
\eqref{stro} also exists since the integrand is regularly varying 
at $0^+$ with index $(\sigma-\vartheta)/p >-1$ by virtue of 
$\sigma>\vartheta-p$. 
By Karamata's Theorem in Appendix~\ref{appe}, \eqref{stro} implies \eqref{sqa}.  
Furthermore, if   
\eqref{doii} holds, then $L_h(\tilde u(r))\sim m_0^\gamma L_h(1/r)$ as $r\to 0^+$ so that  
\eqref{sqa} is refined by 
\begin{equation}
\tilde u(r)\sim \left[m_0^{\gamma-p} M \frac{L_h (1/r) \,L_b (r)}{ L_{\A}(r)} \right]^{-\frac{1}{q-p+1}}r^{-m_0} \quad \text{as } r \to 0^+.
\end{equation}

\begin{remark}
A prototype model for \eqref{doii} is $L_h(t)\sim (\ln t)^\gamma$ as $t\to \infty$, where $\gamma\in \RR$. More generally, \eqref{doii} holds if
$L_h(T)\sim \mathcal L(T)$ as $T\to \infty$ and 
$
\mathcal L(T)=\Pi_{i=1}^k (\ln_{m_i} T)^{\beta_i}$ for $ T>0$ large, where
$k$ and $m_i$  are positive integers and $\beta_i\in \RR$ for every $1\leq i\leq k$. 
We use the notation $\ln_{m_i} $ for the $m_i$-iterated natural logarithm. Without loss of generality, we can take 
$1\leq m_1<m_2<\ldots <m_k$. Then $t\longmapsto L_h(e^t) $ is regularly varying at $\infty$ with index equal to $\beta_1$ (respectively, $0$)
if $m_1=1$ (respectively, $m_1>1$). Similarly, \eqref{doi} 
is verified if  $\left[L_{\mathcal A}(1/T)\right]^{-\frac{q}{p-1}} L_b(1/T)\sim \mathcal L(T)$ as $T\to \infty$. 
\end{remark}

Next, in our second main result, under suitable conditions, we show that there exist positive solutions of 
\eqref{une1} in any of the categories appearing in the complete 
classification of Theorem~\ref{th1.2}. Furthermore, we obtain a uniqueness result for \eqref{une1} subject to a Dirichlet condition on $\partial B_1$ with a prescribed, admissible behaviour at zero.

\begin{theorem}[Existence and uniqueness] \label{wsol}
Let Assumptions $({\mathbf A_1})$--$({\mathbf A_3})$ hold. Assume that $h$ is a non-decreasing function on $(0,\infty)$ and $g \in C^1 (\partial B_1)$ is an arbitrary non-negative function. We consider the following problem
\neweq{prob2} 
\left\{ \begin{aligned}
& {\rm div}\,(\mathcal A(|x|)\, |\nabla u|^{p-2} \nabla u) =b(x) \,h(u)\quad \text{in } B^*:=B_1\setminus\{0\},\\
&  \lim_{|x|\rightarrow 0} \frac{u(x)}{\Phi(x)}=\lambda, \quad u\big|_{\partial B_1}=g, \quad  u>0 \quad \text{in } B^*. 
\end{aligned}
\right.
\endeq
\begin{enumerate}
\item[{\rm (i)}] If $\lambda=0$ and $g \nequiv 0$ on $\partial B_1$, then \eqref{prob2} has a unique solution.
\item[{\rm (ii)}] If $\lambda \in (0,\infty]$, then \eqref{prob2} admits solutions if and only if $b(x)\,h(\Phi) \in L^1(B_{1/2})$. 
\item[{\rm (iii)}] Assume that $b(x)\,h(\Phi) \in L^1(B_{1/2})$ and $h(t)/t^{p-1}$ is non-decreasing for $t>0$.
\begin{enumerate}
\item For $\lambda \in (0,\infty)$, then \eqref{prob2} has a unique solution. The same conclusion holds for $\lambda=\infty$ and $q<q_*$.
\item For $\lambda=\infty$ and $q=q_*$, then \eqref{prob2} has a unique solution provided that either \eqref{doii} or \eqref{doi} holds.
\end{enumerate}
\end{enumerate}
 \end{theorem}

\begin{remark} \label{history2} 
When $b=1$ and $h(t)=|t|^{q-1} t$, our Theorem~\ref{wsol} recovers previous results such as \cite[Theorems 1.2]{fav} (with $\A=1$) and \cite[Theorem 2]{BCCT} (with $p=2$). Moreover, in Theorem~\ref{wsol}, we generalise \cite[Theorem 1.2]{cd} (where $\A=1$) by sharpening the condition under which there exists a unique singular solution to \eqref{prob2}.    
\end{remark}

\begin{remark} \label{equal} In this paper, we focus on the case $p<N+\vartheta$ in Assumption~$({\mathbf A_1})$. We mention that Theorem~\ref{th1.2} and Theorem~\ref{wsol} remain valid also for $p=N+\vartheta$ provided that $\limsup_{r\to 0^+} L_\A(r)<\infty$ (which ensures that $\Phi(r)\to \infty$ as $r\to 0^+$). 
Since $m_2$ in \eqref{stri} becomes zero for $p=N+\vartheta$, we must understand $q_*=\infty$ in connection with the fact that 
$b(x) \,h(\Phi)\in L^1(B_{1/2})$ holds for any $q\in (p-1,\infty)$ and thus in Theorem~\ref{th1.2} only the assertion of (a) is meaningful in which the strong singularity behaviour of (iii) is given by \eqref{stro}.  
\end{remark}

{\em Structure of the paper.} 
We shall always assume $({\mathbf A_1})$--$({\mathbf A_3})$. In Section~\ref{section2}, we apply our main results, specifically on the Examples given by Table~\ref{tab:title}. 
In Section~\ref{sec2} we prove Theorem~\ref{th1.2}(a), which fully classifies the nature of all possible singularities at $0$ for the positive solutions of \eqref{une1} when
$b(x)\,h(\Phi)\in L^1(B_{1/2})$. We emphasise that this is an optimal condition under which, besides weak singularity solutions, there can arise strong singularity solutions of \eqref{une1} (that is $\lim_{|x|\to 0} u(x)/\Phi(x)=\infty$) as stated by Theorem~\ref{wsol} to be proved in Section~\ref{weaksol}. 
The proof of Theorem~\ref{th1.2}(a), and in particular, the analysis of (radial) solutions with strong singularities at $0$ in Theorem~\ref{drum}, represent the crux of this paper. Even in the case $\A=1$, Theorem~\ref{th1.2}(a) is new with regard to the explicit derivation of the asymptotic 
behaviour near $0$ of solutions with strong singularities in the critical case $q=q_*$.  
To establish Theorem~\ref{th1.2}(a), we need to invoke some auxiliary results such as \textit{a priori} estimates, a spherical 
Harnack-type inequality and regularity results, whose proofs are deferred until Section~\ref{basic}. 
In Section~\ref{sec3}, we prove Theorem~\ref{th1.2}(b), which establishes $b(x)\,h(\Phi)\not\in L^1(B_{1/2})$ as a sharp condition such that all positive solutions of \eqref{une1} have a removable singularity at zero, that is, they can be extended as positive continuous solutions of \eqref{une1} in the whole ball $B_1$. 
For the reader's convenience, we gather in Appendix~\ref{appe} the necessary concepts and properties related to the regular variation theory.

\section{Applications} \label{section2}
 
\begin{corollary} \label{mil}
Let Assumptions $({\mathbf A_1})$--$({\mathbf A_3})$ hold. 
Let $\alpha,\beta,\gamma\in \RR$ and $\nu\in (0,1/2)$ be arbitrary. 

\medskip
\begin{minipage}{\linewidth}
\centering
\begin{tabular}{|M{1.5cm}|M{4cm}|M{3cm}|M{3cm}|N}
\hline
\bf    Example & $L_{\A}(r)$ as $r\to 0^+$ & $L_b(r)$ as $r\to 0^+$ & $L_h(t)$ as $t\to \infty$ &\\[10pt]
\hline
    $1$	&  $\displaystyle\left(\ln \frac{1}{r}\right)^\alpha$ &  $\displaystyle\left(\ln \frac{1}{r}\right)^\beta$ 	& $\displaystyle(\ln t)^\gamma$ &\\[20pt]
\hline
    $2$	& $\displaystyle\left(\ln \frac{1}{r}\right)^\alpha$ &  $\displaystyle\left(\ln \frac{1}{r}\right)^\beta$ & $\displaystyle\exp\left\{-(\ln t)^\nu\right\}$ &\\[20pt]
\hline 
 $3$	& $\displaystyle\left(\ln \frac{1}{r}\right)^\alpha \exp\left\{-\frac{p-1}{q} \sqrt{\ln \frac{1}{r}}\right\}$ &  $\displaystyle\exp\left\{-\sqrt{\ln \frac{1}{r}}\right\}$ & $\displaystyle\exp\left\{-(\ln t)^\nu\right\}$	&\\[20pt]
\hline
\end{tabular} \par
\captionof{table}{Examples (corresponding to those in Table~\ref{tab:title})} \label{tab:title2} 
\medskip
\end{minipage}
\begin{enumerate}
\item[{\bf (A)}] If $q<q_*$ in Examples 1--3, then for any positive solution $u$ of \eqref{une1} exactly one of the following holds:  
\begin{enumerate}
\item[{\rm (i)}] $u$ can be extended as a positive continuous solution of \eqref{une1} in $B_1$.
\item[{\rm (ii)}] $u$ has a weak singularity at $0$, that is $\lim_{|x|\to 0}u(x)/\Phi(x)= \lambda\in (0,\infty)$ and, moreover, $u$ satisfies \eqref{b11}. 
\item[{\rm (iii)}] $u$ has a strong singularity at $0$ and, moreover, as $|x|\to 0$, the behaviour of $u$ is given by Table~\ref{tab:title3} below. 

\medskip
\begin{minipage}{\linewidth}
\centering
\begin{tabular}{|M{1.5cm}|M{10cm}|N}
\hline
  \bf  Example &  $u(x) $ is asymptotically equivalent to &\\[10pt]
\hline
    $1$	&  $  \displaystyle |x|^{-m_0} \left[ \frac{m_0^{p-\gamma}}{M}
     \left(\ln \frac{1}{|x|}\right)^{\alpha - \beta - \gamma} \right]^{\frac{1}{q-p+1}}$ &\\[24pt]
\hline
    $2$	& $\displaystyle  |x|^{-m_0}\left[\frac{m_0^p}{M} \left(\ln \frac{1}{|x|}\right)^{\alpha - \beta} \right]^{\frac{1}{q-p+1}}\exp \left\{\frac{1}{q-p+1}\left(m_0\ln \frac{1}{|x|}\right)^\nu \right\} $ &\\[24pt]
\hline 
 $3$	& $ \displaystyle 
  |x|^{-m_0}\left[\frac{m_0^p}{M} 
 \left(\ln \frac{1}{|x|}\right)^{\alpha}\right]^{\frac{1}{q-p+1}}\, 
 \exp \left\{ \frac{1}{q}\left(\ln\frac{1}{|x|}\right)^{\frac{1}{2}}+ \frac{1}{q-p+1}\left(m_0\ln \frac{1}{|x|}\right)^\nu \right\} $ &\\[24pt]
\hline
\end{tabular} \par
\captionof{table}{Strong singularity behaviour for $q<q_*$} \label{tab:title3} 
\medskip
\end{minipage}
\end{enumerate}

\item[{\bf (B)}] If $q=q_*$ (and, in addition, $\alpha q_*/(p-1)>\beta+\gamma  +1$ for Example 1), then the trichotomy in {\bf (A)} remains valid except
(iii) which is replaced by the behaviour in Table~\ref{tab:title4} below.

\medskip
\begin{minipage}{\linewidth}
\centering
\begin{tabular}{|M{1.5cm}|M{11cm}|N}
\hline
  \bf  Example & $ u(x) $  is asymptotically equivalent to &\\[10pt]
\hline
    $1$	& $\displaystyle |x|^{-m_0}
 \left[ 
 \frac{m_0^{p-1-\gamma}\left(\frac{\alpha q_*}{p-1}-\beta-\gamma-1\right)}{m_1}\left(\ln \frac{1}{|x|}\right)^{\alpha -\beta-\gamma-1}
 \right]^{\frac{1}{q-p+1}} $ &\\[28pt]
\hline
    $2$	& $\displaystyle |x|^{-m_0}   \left[ \frac{\nu \,m_0^{p-1+\nu} }{m_1}  \left( \ln \frac{1}{|x|} \right)^{\alpha -\beta+\nu-1}\right]^{\frac{1}{q-p+1}}
 \exp \left\{ \frac{1}{q-p+1} \left(m_0\ln \frac{1}{|x|}\right)^\nu \right\}   $ &\\[28pt]
\hline 
 $3$	& $ \displaystyle |x|^{-m_0}  \left[  \frac{\nu \,m_0^{p-1+\nu} }{m_1}  \left( \ln \frac{1}{|x|} \right)^{\alpha +\nu-1} \right]^{\frac{1}{q-p+1}}\,\exp \left\{  \frac{1}{q}\left(\ln\frac{1}{|x|}\right)^{\frac{1}{2}}+\frac{1}{q-p+1}\left(m_0\ln \frac{1}{|x|}\right)^\nu \right\}  $ &\\[28pt]
\hline
\end{tabular}\par
\captionof{table}{Strong singularity behaviour for $q=q_*$} \label{tab:title4} 
\medskip
\end{minipage}

\item[{\bf (C)}] If $q>q_*$, then any positive solution $u$ of \eqref{une1} can be extended as a positive continuous solution of \eqref{une1} in $B_1$. 
For Example~1, this conclusion also holds for $q=q_*$ and $\alpha q_*/(p-1)\leq \beta+\gamma  +1$.
\end{enumerate}
\end{corollary}

Our next application illustrates how {\em weighted} divergence-form equations such as 
\eqref{une1} arise naturally in the study of 
$p$-Laplacian type equations in \textit{exterior} domains.  

\begin{corollary}
Assuming $2\leq  N\le p<a$ and $q>p-1$, we consider the problem
\neweq{exv}
{\rm div} \left(|\nabla v(\tilde x)|^{p-2} \nabla v(\tilde x)\right) = |\tilde x|^{-a} [v(\tilde x)]^q \quad \text{in } \RR^N \setminus \overline{B_1}.
\endeq
By a modified Kelvin transform where $u(x)=v(\tilde x)$ with 
$x = \tilde x/|\tilde x|^2$ (see \cite[Appendix A]{fp}), 
the behaviour near $\infty$ of the positive solutions of \eqref{exv} 
can be obtained from the behaviour near $0$ of the positive 
solutions of \eqref{une1} 
with $\A(x)=|x|^{2(p-N)}$, $b(x)=|x|^{a-2N}$ and $h(u)=[u(x)]^q$.
Hence, by applying our Theorem~\ref{th1.2}, we find that:
\begin{enumerate}
\item[{\rm (1)}] If $p>N$, then the following classification holds for the positive solutions $v(\tilde x)$ of \eqref{exv}:
 \begin{enumerate}
\item[{\rm (a)}] If $q<\frac{(a-N)(p-1)}{p-N}$, then as $|\tilde x|\to \infty$, exactly one of the following holds 
\begin{enumerate}
\item[{\rm (i)}] $v(\tilde x)$ converges to a positive number;  
\item[{\rm (ii)}] ${|\tilde x|^{-\frac{p-N}{p-1}}}{v(\tilde x)}$ converges to a positive number;
\item[{\rm (iii)}] $|\tilde x|^{-(a-p)/(q-p+1)}v(\tilde x)\to \left[\left(\frac{a-p}{q-p+1}\right)^{p-1}
\left(\frac{-pq+ap-a}{q-p+1}-N\right)\right]^{1/(q-p+1)}$.  
\end{enumerate}
\item[{\rm (b)}] If, in turn, $q\geq \frac{(a-N)(p-1)}{p-N}$, then for every positive solution of \eqref{exv}, only (i) holds.
\end{enumerate}
\item[{\rm (2)}] If $p=N$, then for all $q>p-1$, only (1)(a) holds in which (ii) should read as 
$\lim_{|\tilde x|\to \infty}v(\tilde x)/\ln(|\tilde x|)\in (0,\infty)$.
\end{enumerate}
\end{corollary}

\section{Proof of Theorem~\ref{th1.2}(a): Classification of singularities} \label{sec2}
Let $u$ be any positive solution of \eqref{une1}.
Before proving Theorem~\ref{th1.2}(a), we state some preliminary results 
to be established later in Section~\ref{basic}, under 
Assumptions~$({\mathbf A_1})$--$({\mathbf A_3})$.   
Fix $r_0\in (0,1/2)$. Then the following holds: 

$\bullet$ An {\em a priori} estimate (see Lemma~\ref{apr}): 
There exists a positive constant $C$, depending on $r_0$ such that
\neweq{bst} 
\frac{|x|^p b(x)}{\A(|x|)} \frac{h(u(x))}{[u(x)]^{p-1}}\leq C \quad \text{for every } 0<|x|\leq r_0.
\endeq

$\bullet$ A Harnack-type inequality (see Lemma~\ref{lemharnack}): 
There exists a constant $K>0$ (depending on $p$, $N$ and $r_0$)
such that 
\begin{equation} \label{eqharnack}
\max_{|x|=r} u(x)\leq K \min_{|x|=r}u(x) \quad \text{for all } 0<r\le r_0/2.
\end{equation}

Using \eqref{eqharnack} and the same argument as in \cite[Corollary~4]{BCCT} and \cite[Corollary 4.5]{florica}, the following can be shown:    

\neweq{mac} \left\{ 
\begin{aligned}
& 
\text{If } \limsup_{|x|\rightarrow 0} \frac{u(x) }{ [\Phi (x)]^j}=\infty,\ \text{then } \lim_{|x|\rightarrow 0} \frac{u(x)}{[\Phi(x)]^j}=\infty \quad \text{for } j\in \{0,1\}.\\
& 
\text{If } \liminf_{|x|\to 0} \frac{u(x)}{\Phi(x)}=0,\ \text{then } \lim_{|x|\to 0} \frac{u(x)}{\Phi(x)}=0.
 \end{aligned} \right. 
\endeq

Consequently, we either have $\limsup_{|x|\to 0} u(x)<\infty$ or $\lim_{|x|\to 0} u(x)=\infty$. In the latter case, the {\em a priori} estimate in
\eqref{bst}, together with Assumptions  $({\mathbf A_1})$ and $({\mathbf A_3})$, give that 
\neweq{jam} 
\limsup_{|x|\to 0} \frac{L_b(|x|)}{L_\A(|x|)}\,|x|^{p+\sigma-\vartheta} \,\left[u(x)\right]^{q-p+1} L_h(u(x))<\infty. 
\endeq   
In particular, \eqref{jam} yields that $\limsup_{|x|\to 0}u(x)/T(|x|)<\infty$ 
for some function $T$ regularly varying at $0$ with index $-m_0$. 
Since $\lim_{r\to 0^+} \ln T(r)/\ln (1/r)=m_0$, we find that $\limsup_{|x|\to 0}\ln u(x)/\ln (1/|x|)\leq m_0$. 

\begin{remark} \label{ghj} If $q=q_*$, then $m_0=m_2$ and any positive solution $u$ of \eqref{une1} with a strong singularity at zero satisfies
\neweq{aop} 
\lim_{|x|\to 0} \frac{\ln u(x)}{\ln \left(1/|x|\right)}=m_0. 
\endeq
\end{remark}

$\bullet$ A regularity result (see Lemma~\ref{reg}). 

$\bullet$ If 
$\lim_{|x|\to 0}u(x)/\Phi(x)=0$, then 
$u$ can be extended as a continuous positive solution of \eqref{une1} in $B_1$ (see Lemma~\ref{remcon}).  

\vspace{0.4cm}
{\bf Proof of Theorem~\ref{th1.2}(a).}
Let $b(x)\, h(\Phi)\in L^1(B_{1/2})$ and $u$ be a positive solution of \eqref{une1}.  Let $\lambda:=\limsup_{|x|\to 0} \frac{u(x)}{\Phi(|x|)}$.
Then the categories (i)--(iii) of Theorem~\ref{th1.2}(a) correspond respectively to:
\begin{enumerate}
\item[{\rm (i)}] $\lambda=0$. Then the assertion of (i) in Theorem~\ref{th1.2}(a) follows from Lemma~\ref{remcon}.
\item[{\rm (ii)}] $\lambda\in (0,\infty)$. One can show that $u$ has a weak singularity at $0$ and can verify \eqref{b11} by using the same argument as in \cite[Theorem~5.1]{cd} (see
also \cite[Proposition~6]{BCCT}). We thus omit the details. 
\item[{\rm (iii)}] $\lambda=\infty$. Then \eqref{mac} yields that 
$\lim_{|x|\to 0} u(x)/\Phi(x)=\infty$. We show below how to reduce the proof of (iii) in 
Theorem~\ref{th1.2}(a) to the case of strong singularities for radial
 solutions of an approximate problem \eqref{rd1} treated in Theorem~\ref{drum}. 
We reason as in \cite[Lemma~4.12]{florica}, using Lemmas~\ref{apr} and~\ref{reg} to deduce that 
for every $\varepsilon\in (0,1)$, there exists $r_\varepsilon\in (0,1)$ and 
a function $v_\varepsilon$ satisfying $ \left(1-\varepsilon\right) u\leq v_\varepsilon \leq \left(1+\varepsilon\right) u$ 
in $B_{r_\varepsilon}^*$ with $v_\varepsilon$ a positive solution of 
\begin{equation} \label{rd1}
-{\rm div}\,(\mathcal A(|x|)\, |\nabla v|^{p-2} \nabla v)+|x|^\sigma \, v^q L_b(|x|)\,L_h(v)=0\quad \text{in } B_{r_\varepsilon}^*:=B_{r_\varepsilon}\setminus\{0\}. 
\end{equation}

Moreover, if $v$ is any positive solution of \eqref{rd1}, then as in \cite[Lemma~4]{BCCT}, we can obtain 
two positive {\em radial} solutions of \eqref{rd1} in $B_{r_\varepsilon/2}^*$, say $v_*$ and $v^*$, such that for a sufficiently large constant $K>1$, we have 
\begin{equation} \label{rd2} 
K^{-1} v\leq v_*\leq v\leq v^* \leq K v\quad\text{in } B_{r_\varepsilon/2}^*. 
\end{equation}
We observe that 
any positive radial solution of \eqref{rd1} in 
$B^*$ satisfies 
\begin{equation} \label{radio1}
\frac{d}{dr}\left(r^{N-1+\vartheta} L_{\A}(r) |v'(r)|^{p-2} v'(r) \right)=r^{N-1+\sigma} L_b(r) \,L_h(v(r))\, v^q(r)\quad \text{for } r=|x|\in (0,1).
\end{equation}

In view of \eqref{rd2}, to conclude the assertion of (iii) in Theorem~\ref{th1.2}(a),  
it is enough to prove Theorem~\ref{drum} below.
\end{enumerate}

\begin{theorem} \label{drum}
Let Assumptions $({\mathbf A_1})$--$({\mathbf A_3})$ hold. Suppose that $b(x)\, h(\Phi)\in L^1(B_{1/2})$. Let $v$ be any positive solution of 
\eqref{radio1} with a strong singularity at $0$. 
\begin{enumerate}
\item[{\rm (a)}] If $q<q_*$, then $v(r)\sim \tilde u(r)$ as $r\to 0$, where $\tilde u$ is given by \eqref{stro}.
\item[{\rm (b)}] If $q=q_*$, then assuming either \eqref{doii} or \eqref{doi}, we have 
$v(r)\sim \tilde u(r)$ as $r\to 0$, where $\tilde u$ is given by \eqref{fg1}. 
\end{enumerate}
\end{theorem}

A major advance in this paper compared with C\^irstea and Du \cite{cd} (where $\A=1$) is 
the analysis of the {\em critical case} and the derivation of the asymptotic behaviour of 
the strong singularities. Our contribution here is the development of a perturbation technique suitable for the {\em critical case} $q=q_*$. 
Unlike the subcritical case, where the power model corresponding to $\A=b=1$ and $h(t)=|t|^{q-1} t$ was 
completely understood due to Friedman and V\'eron \cite{fav} (see also Remark~\ref{history}), 
in the critical case we had no model in the literature to provide us with intuition on the asymptotics of strong singularity solutions. As we reveal in our paper, 
the critical case is important in the {\em non-power nonlinearity} case as it represents the threshold between having {\em a trichotomy} classification (as in Theorem~\ref{th1.2}(a))
or {\em no singularities at all} as in Theorem~\ref{th1.2}(b), all depending on whether or not $b(x)\,h(\Phi)$ belongs to $L^1(B_{1/2})$.

The proofs of Theorem~\ref{drum}(a) and (b) are intricate, each being composed of three main steps. 
First, we shall prove here the critical case $q=q_*<\infty$, while also pointing out the major differences between the subcritical and critical cases. 
Under the assumptions of Theorem~\ref{drum}, let $v$ be any positive solution of \eqref{radio1} with a strong singularity at $0$. 
A change of variable $y(s)=v(r)$ with $s=\Phi(r)$ moves the singularity from $r=0$ to $s=\infty$ for the equation
\begin{equation} \label{rado2}
(p-1)\left|y'(s) \right|^{p-2} y''(s)=C_{N,p}^{-p+1}\, r^{N-1+\sigma} L_b(r)\,
L_h(y(s))\left[y(s)\right]^q  \left|\frac{dr}{ds}\right| \ \ \text{for } s\in (0,\infty),
\end{equation}
where, for simplicity, we denote $y'(s)=dy/ds$ and $y''(s)=d^2 y/ds^2$.

\vspace{0.2cm}
{\bf Step~1.} {\em Fix $\eta_0>0$ small. For every $\varepsilon\in (0,1)$ small, there exists $r_\varepsilon\in (0,1)$ such that $\left(1-\varepsilon\right) v_{-\eta}$ and $\left(1+\varepsilon
\right) v_\eta$ is a sub-solution and
super-solution of \eqref{radio1} for $0<r<r_\varepsilon$, respectively, for every $\eta\in [0,\eta_0]$. Moreover, it holds that $\lim_{\eta\to 0^+} v_{\pm \eta}(r)=
\tilde u(r)$ for every $r\in (0,r_\varepsilon]$, where $\tilde u$ is as in Theorem~\ref{drum}.}

\vspace{0.2cm}
The local one-parameter family $v_{\pm \eta}$ of sub- and super-solutions of \eqref{radio1} is constructed such that 
$ v_{\pm \eta}(r)$ converges to $\tilde u(r)$ as $\eta$ approaches $ 0^+ $. The function $\tilde u$ in Theorem~\ref{drum} is regularly varying at $0$ with index $-m_0$, where $m_0$ and $m_2$ are given by \eqref{stri}. 
The definition of $v_{\pm \eta}$ 
in the subcritical case is different from
that of the critical case as follows.

In the subcritical case $q<q_*$, we define $v_{\pm\eta}$ in \eqref{case0} as a regularly 
varying function at $0$ with index $-\left(1\pm \eta\right) m_0$ (here $m_0>m_2$). We shall check the assertion of Step~1 in \S\ref{sssub}.  

In the critical case $q=q_*<\infty$, we have $m_0=m_2$, that is, $\tilde u$ has the same index of regular variation at $0$ as the fundamental solution
$\Phi$ in \eqref{2.1}, namely $-m_2$. In this case, $v_{\pm\eta}$ is defined by \eqref{case1} 
as a regularly varying function at $0$ with index $-m_2$. We shall verify Step~1 in \S\ref{sscrit} with the change of variable $y_{\pm\eta}(s)=v_{\pm \eta}(r)$ where $s=\Phi(r)$.  Notice that when either \eqref{doii} holds or \eqref{doi} holds, by 
the definitions of $\tilde u $ in \eqref{fg1} and $v_{\pm \eta}$ in \eqref{case1}, we infer that 
\neweq{fr1} \lim_{r\to 0^+} \frac{\tilde u(r)}{v_\eta(r)}=0\ \ \text{and} \ \ 
 \lim_{r\to 0^+} \frac{\tilde u(r)}{v_{-\eta}(r)}= \infty\ \ \text{for every } \eta\in [0,\eta_0].  
\endeq

 \vspace{0.2cm}
 {\bf Step 2.} {\em The functions } $v_\eta$ and $v_{-\eta}$ {\em constructed in Step~1 satisfy the following property:} 
 \neweq{mot}
 \lim_{r\rightarrow 0^+} \frac{v(r)}{v_{\eta}(r)}=0 \quad \text{and} \quad \lim_{r\rightarrow 0^+} \frac{v(r)}{v_{-\eta}(r)}= \infty.
\endeq
 
 In both the subcritical and critical cases, since $v$ has a strong singularity at $0$, that is 
$v(r)/\Phi(r)\to \infty$ as $r\to 0^+$, then we have 
$y(s)/s\to \infty$ as $s\to \infty$. 
Using that $y''(s)\geq 0$, we find that 
$y'(s)$ is increasing so that $\lim_{s\to \infty} y'(s)=\infty$. As the function $s\longmapsto s y'(s)-y(s)$ is increasing on $(0,\infty)$ and $\lim_{s\to \infty} y(s)=\infty$, we see that 
\begin{equation} \label{liminf} 
\liminf_{s\to \infty} \frac{s y'(s)}{y(s)}\geq 1.
\end{equation}

In the {\em subcritical} case, we shall use \eqref{liminf} in Lemma~\ref{fact1}(b) of \S\ref{sssub} to improve the behaviour of the solution
$v$ of \eqref{radio1} from dominating near zero the fundamental solution $\Phi$ (of index $-m_2$) to dominating
{\em any} function $f$ regularly varying at zero with index $-\kappa$, where $m_2<\kappa<m_0$. 
We deduce \eqref{mot} by using Lemma~\ref{fact1} with $f=v_{\pm \eta}$ since the index of regular variation at $0$ for the function $v_{\eta}$ (respectively, $v_{-\eta}$) is smaller (respectively, bigger) than $-m_0$. We point out that Lemma~\ref{fact1} relies essentially on the assumption that $q<q_*$ and cannot be adapted to the critical case. 
 
Hence, in the {\em critical} case, we need a new argument that takes into account that $v_{\pm \eta}$ varies regularly at $0$ with the same index as $\tilde u$. We now prove Step~2 in the critical case.

\vspace{0.2cm}
{\em Proof of Step~2 for the critical case $q=q_*$.} 
 
\vspace{0.2cm}
The main ingredient in the proof of \eqref{mot} is given by the following 
 \neweq{aur10} 0<\liminf_{r\to 0^+} \frac{v(r)}{\tilde u(r)}\leq \limsup_{r\to 0^+} \frac{v(r)}{\tilde u(r)}<\infty. \endeq

By combining \eqref{fr1} and \eqref{aur10}, we conclude \eqref{mot} in the critical case. 
 
 \vspace{0.2cm} 
{\em Proof of \eqref{aur10}.}  
Using \eqref{jam} and \eqref{liminf}, we infer that $\limsup_{s\to \infty} sy''(s)/y'(s)<\infty$. Indeed,   
by \eqref{rado2}, we have
\begin{equation} \label{loc1}
 \frac{sy''(s)}{y'(s)} = \frac{1}{p-1} \left[\frac{y(s)}{sy'(s)}\right]^{p-1}[\Upsilon(r)]^{-p}\,
 \frac{L_b(r)}{L_{\A}(r)}\, r^{p+\sigma-\vartheta} 
L_h(y(s)) \left[y(s)\right]^{q_*-p+1},  
\end{equation} where $\Upsilon$ is given by \eqref{kphi}. 
For $s_0>0$, there exists a large constant $C>0$ so that $s\mapsto s y'(s)-Cy(s)$ is non-increasing for all $s>s_0$. 
It follows that $\ell=\limsup_{s\to \infty} s y'(s) / y(s) <\infty$.  
From \eqref{liminf}, we can take $s_0>0$ large such that 
\begin{equation} \label{lo1} 
\frac{1}{2}\leq \frac{s \,y'(s)}{y(s)}\leq 2\ell \quad \text{for all } s\geq s_0.
\end{equation} 
In view of Remark~\ref{ghj}, we find that $\ln y(s)\sim \ln s$ as $s\to \infty$. Consequently, as $s\to \infty$,
we obtain that 
\neweq{tynn} 
\left\{ \begin{aligned}
& m_0^\gamma L_h(1/r) \sim L_h(s)\sim L_h(y(s)) && \text{if } \eqref{doii} \ \text{holds};&\\
& (m_0)^{-j} \left[L_{\A}(1/y(s))\right]^{-\frac{q_*}{p-1}} L_b(1/y(s))\sim 
\left[L_{\mathcal A} (\Phi^{-1}(s))\right]^{-\frac{q_*}{p-1}} L_b(\Phi^{-1}(s))   
&& \text{if } \eqref{doi} \ \text{holds}.& 
\end{aligned} \right. \endeq
For all $s\geq s_0$, by using \eqref{lo1} and \eqref{tynn} in \eqref{rado2}, we find positive constants $c_1$ and $c_2$ so that 
\neweq{newg} 
\left\{ \begin{aligned}
 c_1 r^{N-1+\sigma} L_b(r) \, h(\Phi(r)) & \leq \left[y'(s)\right]^{-q_*+p-2}
y''(s) \left| \frac{ds}{dr}\right| \leq c_2  r^{N-1+\sigma} L_b(r)\, h(\Phi(r)) && \text{if } \eqref{doii} \ \text{holds};&\\
  c_1 \frac{d}{ds} [F(1/y(s))] & \leq \left[y'(s)\right]^{-q_*+p-2}
y''(s) \leq c_2  \frac{d}{ds} [F(1/y(s))] && \text{if } \eqref{doi} \ \text{holds},& 
\end{aligned} \right. \endeq
where $F$ is defined by \eqref{fodef}.

\vspace{0.2cm}
{\bf Case 1:} {\em Assume that \eqref{doii} holds.} 

\vspace{0.2cm}

Since $y'(s)\to \infty$ as $s\to \infty$, by integrating \eqref{newg}, we obtain that
\begin{equation} \label{bob} 
c_3 F(\Phi^{-1}(s))\leq 
\left[ y'(s)\right]^{-q_*+p-1} \leq c_4 F(\Phi^{-1}(s)) \quad \text{for all } s> s_0,
\end{equation} 
where $c_3$ and $c_4$ are positive constants.
Using \eqref{lo1} in \eqref{bob}, then reversing the change of variable $y(s)=v(r)$ with $s=\Phi(r)$, we infer that there exist positive constants $c_5$ and $c_6$ such that 
\begin{equation}
 c_5 \left[F(r) \right]^{-\frac{1}{q_*-p+1}} \Phi(r) \leq v(r)
\leq c_6 \left[ F(r) \right]^{-\frac{1}{q_*-p+1}}  \Phi(r) \quad \text{for all } r\in (0,\Phi^{-1}(s_0)).
\end{equation}
Hence, using \eqref{kphi} and the definition of $\tilde u$ in \eqref{fg1}, we conclude Step~2 in Case 1. 

\begin{remark} \label{remark6}
Notice that when \eqref{doii} holds, the existence of a solution  $v$ of \eqref{radio1} with a strong singularity at zero implies that $b(x)\,h(\Phi(|x|)) \in L^1(B_{1/2})$. Indeed, fixing $r_0 \in (0, \Phi^{-1}(s_0))$, then for every $\varepsilon \in (0,r_0)$, by integrating the first inequality in \eqref{newg} with respect to $r$ from $\varepsilon$ to $r_0$, and letting $\varepsilon \to 0$, we conclude the claim (using Remark~\ref{remark2}). A more general statement is proven later in Lemma~\ref{radsy}.
\end{remark}

{\bf Case 2:} {\em Assume that \eqref{doi} holds.}  

\vspace{0.2cm}

By twice integrating \eqref{newg}, we find positive constants $c_3$ and $c_4$ such that 
$$ c_3 \leq 
\frac{d}{ds} \left( \int_{y(s_0)}^{y(s)} \left[F(1/t)\right]^{\frac{1}{q_*-p+1}} dt\right) \leq
 c_4 \quad \text{for every } s>s_0.
$$
We thus conclude that 
$$ 0<\liminf_{s\to \infty} \frac{ \int_{y(s_0)}^{y(s)} \left[F(1/t)\right]^{\frac{1}{q_*-p+1}} dt }{s} \leq
\limsup_{s\to \infty} \frac{ \int_{y(s_0)}^{y(s)} \left[F(1/t)\right]^{\frac{1}{q_*-p+1}} dt }{s}<\infty.
$$ 
This, jointly with \eqref{kphi} and the definition of $\tilde u$ in \eqref{fg1}, proves the assertion of Step~2 in Case~2.

  \vspace{0.2cm}
  {\bf Step~3.} {\em Proof of Theorem~\ref{drum} concluded.}
  
   \vspace{0.2cm}
  {\em Proof of Step~3.} The reasoning is the same for the subcritical and critical case. It is based on the previous two steps and the following comparison principle
   to be used frequently in the paper.

\begin{lemma}[Comparison principle, see Theorem 2.4.1 in \cite{PS}] \label{cp}
Let $\Omega$ be a bounded domain in $\RR^N$ with $N\geq 2$. 
Let $u,v\in C^1(\Omega)$ satisfy (in the sense of distributions in $\mathcal D'(\Omega)$) the pair of differential inequalities 
$$
- {\rm div}\, \boldsymbol{A} (x, \nabla u) + B(x,u) \leq 0 \quad \text{and} \quad
- {\rm div}\,  \boldsymbol{A} (x, \nabla v) + B(x,v) \geq 0\quad \text{in } \Omega.
$$
Suppose that $ \boldsymbol{A} : \Omega \times \RR^N \rightarrow \RR^N$ is in $L_{\rm loc}^\infty (\Omega \times \RR^N)$ and $B: \Omega \times \RR \rightarrow \RR$ is in $L_{\rm loc}^\infty (\Omega \times \RR)$ such that $B=B(x,z)$ is independent of $\boldsymbol{\xi}$ and non-decreasing in $z$, whereas
$ \boldsymbol{A}= \boldsymbol{A}(x,\boldsymbol{\xi})$ 
is independent of $z$ and 
monotone in $ \boldsymbol{\xi}$, that is
$$ \langle  \boldsymbol{A}(x,\boldsymbol{\xi})-  \boldsymbol{A}(x,\boldsymbol{\eta}),\boldsymbol{\xi}-\boldsymbol{\eta}\rangle >0\quad \text{when } 
\boldsymbol{\xi} \not=\boldsymbol{\eta}.
$$

If $u \leq v$ on $\partial \Omega$, then $u \leq v$ in $\Omega$.
\end{lemma}

   Let $\varepsilon\in (0,1)$ be small and $r_\varepsilon\in (0,1)$ be as in Step~1. Fix $\eta\in [0,\eta_0]$ arbitrarily.  
Then, $\left(1+\varepsilon\right)v_\eta(r) + v(r_\varepsilon)$ and $v(r)+\tilde u(r_\varepsilon)$ are super-solutions of \eqref{radio1} for $r\in (0,r_\varepsilon)$. 
By \eqref{mot} and Lemma~\ref{cp}, we have
\neweq{yy} 
v(r)\leq \left(1+\varepsilon\right) v_\eta(r)+v(r_\varepsilon) \quad \text{and}\quad 
\left(1-\varepsilon\right) v_{-\eta}(r) \leq v(r)+\tilde u(r_\varepsilon) \quad \text{for all } 0<r\leq r_\varepsilon. 
\endeq 
Since $r_\varepsilon$ is independent of $\eta\in [0,\eta_0]$, by letting $\eta\to 0^+$ in \eqref{yy}, we find that
\neweq{yy1} 
v(r)\leq \left(1+\varepsilon\right) \tilde u(r)+v(r_\varepsilon) \quad \text{and} \quad
\left(1-\varepsilon\right) \tilde u(r)\leq v(r) +\tilde u(r_\varepsilon) \quad \text{for all } 0<r\leq r_\varepsilon. 
\endeq 
By letting $r\to 0^+$ in \eqref{yy1}, we deduce that 
\neweq{yy2}
1-\varepsilon\leq \liminf_{r\to 0^+} \frac{v(r)}{\tilde u(r)}\leq 
 \limsup_{r\to 0^+} \frac{v(r)}{\tilde u(r)}\leq 1+\varepsilon. 
\endeq 
Finally, by passing to the limit $\varepsilon\to 0^+$ in \eqref{yy2}, we conclude 
 that $v(r)\sim \tilde u(r)$ as $r\to 0^+$.

\subsection{Proof of Step~1 in the critical case $q=q_*$ of Theorem~\ref{drum}} \label{sscrit}

In this subsection, it remains only for us to 
establish the claim of Step~1 as outlined in the proof of Theorem~\ref{drum}. 
We first give the construction of a local family of sub- and super-solutions of \eqref{radio1}. 
Let $F$ be given by \eqref{fodef} and
$c>0$ be a large constant. Fix $\eta_0\in (0,1)$ small.  
Then for any $\eta\in [0,\eta_0]$, we define $v_{\pm \eta}(r)$ for $r>0$ small, as follows
\neweq{case1}
\left\{ \begin{aligned}
& v_{\pm\eta}(r):=C_{N,p}^{-1} \left(\frac{m_1 m_0^{\gamma-q}}{1\pm \eta}\right)^{-\frac{1}{q-p+1}}  \int_c^{\Phi(r)} 
\left[F(\Phi^{-1}(t))\right]^{-\frac{1\pm \eta}{q-p+1}} dt
&& \text{if }  \eqref{doii} \ \text{holds},\\
&  \int_c^{v_{\pm \eta}(r)} [F(1/t)]^{\frac{1\pm \eta}{q_*-p+1}} \,dt= C_{N,p}^{-1} \left(\frac{m_1m_0^{-q-1-j}}{1\pm\eta}\right)^{-\frac{1}{q-p+1}}\Phi(r) && \text{if } \eqref{doi} \ \text{holds}.
\end{aligned} \right. 
\endeq

We set $y_{\pm \eta}(s)=v_{\pm \eta}(r)$ with $s=\Phi(r)$. Using $y_{\pm \eta}'(s)$ and $y_{\pm \eta}''(s)$ to denote $dy_{\pm \eta}/ds$ and $d^2y_{\pm \eta}/ds^2$, respectively, then 
\neweq{sase}
(p-1)\left(y_{\pm \eta}'(s)\right)^{p-2} y_{\pm \eta}''(s)=
\frac{1}{m_1} \left( y_{\pm \eta}'(s)\right)^{q_*}\left|  \frac{d }{ds}  \left[ \left(y_{\pm \eta}'(s) \right)^{-q_*+p-1} \right]\right|.  \endeq

{\bf Step~1.} {\em For every $\varepsilon\in (0,1)$ small, there exists $s_\varepsilon>0$ large such that 
$\left(1-\varepsilon\right) y_{-\eta}$ and $\left(1+\varepsilon\right) y_\eta$ is a sub-solution and
super-solution of \eqref{rado2} for $s>s_\varepsilon$, respectively, for every $\eta\in [0,\eta_0]$.} 

\vspace{0.2cm}

From \eqref{case1}, we find that 
 \neweq{dify}
y_{\pm \eta}'(s)=
 \left\{ \begin{aligned}
 & C_{N,p}^{-1} \left(\frac{m_1 m_0^{\gamma-q_*}}{1\pm \eta}\right)^{-\frac{1}{q_*-p+1}}  \left[F(r)\right]^{-\frac{1\pm \eta}{q_*-p+1}}
 && \text{if }  \eqref{doii}\ \text{holds},&\\
 & C_{N,p}^{-1} \left(\frac{m_1m_0^{-q_*-1-j}}{1\pm\eta}\right)^{-\frac{1}{q_*-p+1}} \left[F
 (1/y_{\pm \eta}(s))\right]^{-\frac{1\pm \eta}{q_*-p+1}}
  && \text{if }  \eqref{doi}\ \text{holds}.&
 \end{aligned} \right.
\endeq
Moreover, we obtain 
the following asymptotic equivalence (uniform with respect to $\eta$)
\neweq{doiii} \ln y_{\pm \eta}(s) \sim \ln s\quad \text{and} \quad  s y_{\pm \eta}'(s) \sim y_{\pm \eta}(s)\ \  \text{as } s\to \infty. \endeq
From \eqref{doiii}, we deduce the following asymptotic equivalence as $s\to \infty$ (uniform with respect to $\eta$)
\neweq{tyn} 
\left\{ \begin{aligned}
& m_0^\gamma L_h(1/r) \sim L_h(s)\sim L_h(y_{\pm \eta}(s)) && \text{if } \eqref{doii} \ \text{holds};&\\
& (m_0)^{-j} \left[L_{\A}(1/y_{\pm \eta}(s))\right]^{-\frac{q_*}{p-1}} L_b(1/y_{\pm \eta}(s))\sim 
\left[L_{\mathcal A} (\Phi^{-1}(s))\right]^{-\frac{q_*}{p-1}} L_b(\Phi^{-1}(s))   
&& \text{if } \eqref{doi} \ \text{holds}.& 
\end{aligned} \right. \endeq
We introduce the notation $\mathcal K_{\pm \eta}(s):=\dfrac{\Upsilon(r)}{m_0} \,
 \dfrac{s\,y_{\pm \eta}'(s)}{y_{\pm \eta}(s)}$, where $\Upsilon$ is given by \eqref{kphi}. 
 We also denote $R_{\pm \eta}(s)$ as follows
\neweq{reta}
 R_{\pm \eta}(s)=\left\{ 
 \begin{aligned}
& \frac{m_0^\gamma \,L_h(1/r)}{L_h(y_{\pm \eta}(s))} \,\left[F(r)\right]^{\pm \eta} \left[\mathcal K_{\pm \eta}(s)\right]^{q_*}
&& \text{if } \eqref{doii} \ \text{holds},&\\ 
& m_0^{-j} \left[ \frac{L_\A (1/y_{\pm \eta}(s))}{L_\A (\Phi^{-1}(s))}\right]^{-\frac{q_*}{p-1}}\,
\frac{L_b (1/y_{\pm \eta}(s))}{L_b (\Phi^{-1}(s))} 
\left[F(1/y_{\pm \eta}(s))\right]^{\pm \eta} \left[\mathcal K_{\pm \eta}(s)\right]^{q_*+1}
&& \text{if } \eqref{doi} \ \text{holds}.&
\end{aligned} \right.
\endeq
Since $m_0=m_2$ for $q=q_*$, using \eqref{kphi} and \eqref{doiii}, we infer that $\lim_{s\to \infty}\mathcal K_{\pm \eta}(s)= 1$
 uniformly with respect to $\eta$. 
Hence, using \eqref{tyn}, we derive the following asymptotics as $s\to \infty$ 
 (uniform with respect to $\eta$)
\neweq{reta1}
 R_{\pm \eta}(s)\sim \left\{ 
 \begin{aligned}
& \left[F(r)\right]^{\pm \eta}
&& \text{if } \eqref{doii} \ \text{holds},&\\ 
& 
\left[F(1/y_{\pm \eta}(s))\right]^{\pm \eta} 
&& \text{if } \eqref{doi} \ \text{holds}.&
\end{aligned} \right.
\endeq
The right-hand side of \eqref{sase} equals the product between $R_{\pm \eta}(s)$ and the right-hand side of \eqref{rado2} for
$y=y_{\pm \eta}$. By the definition of $F$ in \eqref{fodef}, we have $ \lim_{r\to 0^+}F(r)=0$. 
Since $q>p-1$, using \eqref{reta1}, we conclude Step~1.

\subsection{Proof of Steps~1 and 2 in the subcritical case $q<q_*$ of Theorem~\ref{drum}} \label{sssub}

We need only to justify the first two steps in the outline of the proof of Theorem~\ref{drum}. We shall adapt the perturbation method
initiated by C\^irstea and Du in \cite{cd}. 
We construct a local family of sub-and super-solutions of \eqref{radio1}. 
Fix $\eta_0\in (0,1)$ such that $ 2\eta_0 (p-1)M<1$, where $M$ is the positive constant given by \eqref{stro}.  
For every $\eta\in [0,\eta_0]$, we define the function $v_{\pm \eta}$ and the constant $C_{\pm \eta}>0$ as
\neweq{case0}
 v_{\pm \eta}(r)=C_{\pm \eta}[\tilde u(r)]^{1\pm \eta} \ \text{ for } r\in (0,1) \quad \text{where } C_{\pm \eta}^{q-p+1}:=  (1\pm \eta)^{p-1} \left[1\pm \eta M(p-1)\right].
\endeq  
From this definition, we have that $\lim_{\eta\to 0^+} v_{\pm \eta}(r)=
\tilde u(r)$ for every $r\in (0,1)$ and $\lim_{\eta\to 0}C_{\pm \eta}=1$.

\vspace{0.2cm}
{\bf Step~1.} {\em 
For every $\varepsilon\in (0,1)$ small, there exists $r_\varepsilon\in (0,1)$ such that $\left(1-\varepsilon\right) v_{-\eta}$ and $\left(1+\varepsilon
\right) v_\eta$ is a sub-solution and
super-solution of \eqref{radio1} for $0<r<r_\varepsilon$, respectively, for every $\eta\in [0,\eta_0]$.}  
\vspace{0.2cm}

{\sc Claim}. {\em We see that $\tilde u$ satisfies \eqref{radio1} asymptotically as $r\to 0^+$.} 

\vspace{0.2cm}
{\em Proof of Claim.}
Let $r_0\in (0,1)$ be small so that $\tilde u(r_0)>t_0$, where $t_0$ is as in Remark~\ref{asb}. 
For all $r\in (0,r_0)$, we set  
\neweq{qqq}
\left\{
\begin{aligned}
 Q_{\pm \eta}(r)&:=r^{N-1+\vartheta} L_{\A}(r) \left|v_{\pm \eta}'(r)\right|^{p-2} v_{\pm \eta}'(r),\\
  P(r)&:=M\left[q+1+\frac{\tilde u(r) \,L_h'(\tilde u(r))}{L_h(\tilde u(r))}-\frac{\tilde u(r)\,\tilde u''(r)}{\left[\tilde u'(r)\right]^2}
+ \left(N-1+\sigma +\frac{rL_b'(r)}{L_b(r)} \right)\frac{\tilde u(r)}{r\,\tilde u'(r)}
\right].
\end{aligned} \right.
\endeq
One can verify that $\lim_{r\to 0^+}P(r)=1$ using the definition of $M$ in \eqref{stro}.  
By differentiating \eqref{stro}, we find that 
\neweq{kitt0}
Q_0(r)=M r^{N-1+\sigma} L_b(r) \frac{\left[\tilde u(r)\right]^{q+1}}{\tilde u'(r)}\,L_h(\tilde u(r))
\quad \text{for all } r\in (0,r_0).
\endeq
The claim follows since $Q_0'(r)$ equals the product between $P(r)$ in \eqref{qqq} and the right-hand side of 
\eqref{radio1} for $v=\tilde u$.

\vspace{0.2cm}
By twice differentiating \eqref{case0}, we obtain that 
\neweq{kitt1}
\left\{ \begin{aligned}
Q_{\pm \eta}(r)&=\left[C_{\pm \eta}(1\pm \eta)\right]^{p-1} [\tilde u(r)]^{\pm \eta (p-1)} \,Q_0(r),\\
\frac{dQ_{\pm \eta}}{dr}&=\left[C_{\pm \eta}(1\pm \eta)\right]^{p-1} [\tilde u(r)]^{\pm \eta (p-1)}
\left\{ \pm \eta\,(p-1)M\left[\tilde u(r)\right]^q L_h(\tilde u(r)) \,L_b(r) \,r^{N-1+\sigma} +\frac{dQ_{0}}{dr}\right\}.
\end{aligned}\right.
\endeq
Hence, using \eqref{case0} and the above claim, we find the following asymptotics (uniform with respect to $\eta$)
\neweq{nume} 
\frac{dQ_{\pm \eta}}{dr}\sim C_{\pm \eta}^q r^{N-1+\sigma} \,L_b(r)\, L_h(\tilde u(r))\,  [\tilde u(r)]^{q\pm \eta (p-1)}  \quad \text{as } r\to 0^+. 
\endeq
From Remark~\ref{asb} in Appendix~\ref{appe}, the function $t\longmapsto t^{q-p+1}\,L_h(t) $ is increasing on $(0,\infty)$ so that  
$$ L_h(\tilde u^{1-\eta}) \,[\tilde u(r)]^{-\eta\,\left(q-p+1\right)}\leq L_h(\tilde u(r)) \leq L_h(\tilde u^{1+\eta}) \,[\tilde u(r)]^{\eta\,\left(q-p+1\right)}
$$
for every $r\in (0,r_0)$ and all $\eta\in [0,\eta_0]$. 
This, together with \eqref{case0}, implies that for every $r\in (0,r_0)$ and all $\eta\in [0,\eta_0]$
\neweq{mnb} \pm C_{\pm \eta}^q\,  L_h(\tilde u(r)) \left[
 \tilde u(r)\right]^{q\pm \eta(p-1)}\leq \pm L_h( v_{\pm \eta}(r)/C_{\pm \eta}) \left[v_{\pm \eta}(r)\right]^q.
\endeq
Since $q>p-1$, from \eqref{nume}, \eqref{mnb} and Proposition~\ref{pn} in Appendix~\ref{appe}, we conclude the proof of Step~1. 

\vspace{0.2cm}
{\bf Step~2.} {\em Any positive solution $v$ of \eqref{radio1} with a strong singularity at $0$ satisfies} \eqref{mot}. 
\vspace{0.2cm}

Since
$v_{\pm\eta}$ is regularly varying at $0$ with index $-\left(1\pm \eta\right) m_0$, 
we conclude Step~2 based on Lemma~\ref{fact1} with $f=v_{\pm \eta}$.

\begin{lemma} \label{fact1}
Let $({\mathbf A_1})$--$({\mathbf A_3})$ hold and $q<q_*$.   
Suppose that $v$ is a positive solution  of \eqref{radio1} with a strong singularity at zero. 
Let $f $ be a regularly varying function at zero with real index $-\kappa$. With $m_0$ given by \eqref{stri}, the following hold:
\begin{enumerate}
\item[{\rm (a)}] If $\kappa>m_0$, then 
$\lim_{r\to 0^+} v(r)/f(r)=0$.
\item[{\rm (b)}] If $\kappa < m_0$, then 
$\lim_{r\to 0^+} v(r)/f(r)=\infty$. 
\end{enumerate}
\end{lemma}

\begin{proof}
We adapt ideas from C\^irstea and Du \cite[Theorem~1.4]{cd}. 

\vspace{0.2cm}
(a) The \textit{a priori} estimates in \eqref{bst} (see Lemma~\ref{apr} for a proof) show that $v$ 
is bounded from above near zero by a regularly varying function at $0$
with index $-m_0$. The assertion now follows easily since every regularly varying function
at $0$ with positive (respectively, negative) index must converge to $0$ (respectively, $\infty$). 

\vspace{0.2cm}
(b) Since $\kappa<m_0$,  we can choose $q_1\in (q,q_*)$ sufficiently close to $q$ such that $\kappa <(p+\sigma-\vartheta)/(q_1-p+1)$. 
Then, $\lim_{t\to \infty}t^{q-q_1} L_h(t) =0$ (see Remark~\ref{rlimo} in Appendix~\ref{appe}) and using  \eqref{liminf}, we can let $s_0>0$ large and find that
\neweq{fjp} 
L_h(y(s))\, [y(s)]^q\leq [y(s)/2]^{q_1}\leq  s^{q_1}[y'(s)]^{q_1}\quad \text{for all } s\geq s_0. 
\endeq 
We set $f_{q_1}(r):=r^{N-1+\sigma} L_b(r) [\Phi(r)]^{q_1}$ for $r\in (0,1)$. Since $\Phi$ is regularly varying at $0$ with index $-m_2$ (see \eqref{kphi}), we find that $f_{q_1}$ is regularly varying at $0$ with index $N+\sigma-q_1 m_2-1$, which is greater than $-1$. This gives that $\int_{0^+} f_{q_1}(\xi)\,d\xi <\infty$. Moreover, the function
$F_{q_1}(r)= \int_r^{\Phi^{-1}(s_0)}  \left[\int_0^{\tau} f_{q_1}(\xi)\,d\xi
\right]^{-\frac{1}{q_1-p+1}} |\Phi'(\tau)|\, d\tau$ is regularly varying at zero with index
$-(p+\sigma-\vartheta)(q_1-p+1)$, which is less than $-\kappa$ from our choice of $q_1$. We thus have $\lim_{r\to 0^+} F_{q_1}(r)/f(r)=\infty$.

We conclude that $\lim_{r\to 0^+} v(r)/f(r)=\infty$ by showing that $\liminf_{r\to 0^+} v(r)/F_{q_1}(r)>0$. Indeed, we see that 
\begin{equation} \label{faq}
\liminf_{r\to 0^+} \frac{v(r)}{ F_{q_1}(r)} 
= \liminf_{s\to \infty} \frac{y(s)}{\int_{s_0}^s \left[  \int_0^{\Phi^{-1}(t)} f_{q_1}(\xi)\,d\xi
\right]^{-\frac{1}{q_1-p+1}} \,dt}. 
\end{equation}
From \eqref{rado2} and \eqref{fjp},
we deduce that 
\neweq{pomm} \left[y'(s)\right]^{p-2-q_1} y''(s)\leq -\frac{C_{N,p}^{-p+1}}{p-1} f_{q_1}(\Phi^{-1}(s))\,\frac{d (\Phi^{-1}(s))}{ds}\quad \text{for all } s>s_0. 
\endeq
Recall that $\lim_{s\to \infty} y'(s)=\infty$ since $v$ has a strong singularity at $0$. 
Thus, by integrating \eqref{pomm}, we obtain that 
$$ 
y'(s)\geq \left[ \frac{\left(q_1-p+1\right) C_{N,p}^{-p+1}}{p-1}  \int_0^{\Phi^{-1}(s)} f_{q_1}(\xi)\,d\xi
\right]^{-\frac{1}{q_1-p+1}} \quad \text{for all } s> s_0, 
$$ which shows that the right-hand side of \eqref{faq} is positive. 
This concludes the assertion of Lemma~\ref{fact1}(b). 
\end{proof}

\section{Proof of Theorem~\ref{th1.2}(b): Removability of singularities} \label{sec3}
Throughout this section, we let Assumptions $({\mathbf A_1})$--$({\mathbf A_3})$ hold. The proof of Theorem~\ref{th1.2}(b) relies on two main ingredients, whose verification is postponed to the end of this section.

\begin{lemma} \label{remcon}
If $u$ is a positive solution of \eqref{une1} such that $\lim_{|x|\to 0}u(x)/\Phi(x)=0$, then there exists $\lim_{|x|\to 0}u(x)\in (0,\infty)$ and $\lim_{|x|\to 0} |x||\nabla u(x)|=0$. Moreover, 
$u$ can be extended as a continuous positive solution of \eqref{une1} in $B_1$. 
\end{lemma}

This result, which was also invoked in the proof of Theorem~\ref{th1.2}(a)(i), generalises \cite[Lemma~3.2(ii)]{cd} (where $\A=1$) and \cite[Proposition~3]{BCCT} (where $p=2$, $b=1$ and $h(u)= u^q$).

\begin{lemma} \label{radsy} 
If \eqref{radio1} has a positive solution 
with either a weak or a strong singularity at $0$, then $b(x)\, h(\Phi)\in L^1(B_{1/2})$. 
\end{lemma}

We show how to use Lemma~\ref{remcon} and Lemma~\ref{radsy} to finish the proof of Theorem~\ref{th1.2}(b). 
We thus assume that $b(x)\, h(\Phi)\not\in L^1(B_{1/2})$ and prove that any positive solution of \eqref{une1} can be extended as a positive solution of \eqref{une1} in $B_1$. By Remark~\ref{remark2}, we have $q\geq q_*$, with $q_*$ as in
\eqref{critex}. Our argument is twofold:

\vspace{0.2cm}
{\bf Case 1:} $q>q_*$. 

\vspace{0.2cm}
Since $m_0<m_2$, the claim follows from Lemma~\ref{remcon} and the \textit{a priori} estimates in 
\eqref{jam}. Indeed, we have $\limsup_{|x|\to 0} u(x)/T(|x|)<\infty$ for a function $T$ regularly varying at $0$ with index $-m_0$.
Using that $\Phi\in RV_{-m_2}(0+)$, by Remark~\ref{rlimo} and Definition~\ref{def-reg} in Appendix~\ref{appe}, we find that 
$\lim_{r\to 0^+} T(r)/\Phi(r)=0$ so that $\lim_{|x|\to 0} u(x)/\Phi(x)=0$ for any positive solution $u$ of \eqref{une1}. 
Then, by Lemma~\ref{remcon}, we conclude the proof of Theorem~\ref{th1.2}(b).

\vspace{0.2cm}
{\bf Case 2:} $q=q_*$. 

\vspace{0.2cm}
The previous argument no longer applies since $T$ and $\Phi$ are now regularly varying 
at $0$ with the same index $-m_0$. Hence, $T/\Phi$ is slowly varying at $0$, whose 
behaviour at $0$ is, in general, undetermined as illustrated by Example~\ref{example1} 
in Appendix~\ref{appe}. In view of Lemma~\ref{remcon}, 
we conclude the proof by showing that $\lim_{|x|\to 0} u(x)/\Phi(x)=0$.

Assuming the contrary and using \eqref{mac}, we deduce $\lim_{|x|\to 0} u(x)=\infty$. Then there exists $k \in (0,1/2)$ and a positive solution $v_*$ of \eqref{radio1} for $0<r<k$ 
such that $C_1 u\leq v_*\leq C_2$ in $B^{*}_k$,
where $C_1$ and $C_2$ are positive constants. 
Thus, by Lemma~\ref{radsy}, we cannot have $\limsup_{|x|\to 0} u(x)/\Phi(x)\in (0, \infty]$. This completes the proof of Theorem~\ref{th1.2}(b).

\begin{proof}[Proof of Lemma~\ref{remcon}]
Let $u$ be a positive solution of \eqref{une1} such that $\lim_{|x|\to 0}u(x)/\Phi(x)= 0$. For convenience, we define
 $$\theta:=\limsup_{|x| \rightarrow 0} u(x). $$
By the comparison principle (Lemma~\ref{cp}), we find as in \cite[Lemma~3.2]{cd} that $\theta<\infty$. Since \eqref{gr} fails for our general assumption ({$\mathbf A_1$}), we cannot invoke
 \cite[Theorem~1]{serrin65} to conclude the proof, unlike the case $\A=1$ treated in \cite{cd}.

\vspace{0.2cm}
We show below that $\theta>0$. 
In the special case $p=2$ and $h(t)=t^q$ of \cite{BCCT}, the claim follows by a reduction to radial solutions, coupled with a change of variable and 
\cite[Theorem~1.1]{tal}. For our general divergence-form equation, we require different ideas that are inspired by
\cite[Lemma~5.2]{florica}. 

Since Assumptions $({\mathbf A_1})$--$({\mathbf A_3})$ hold and $\theta<\infty$, there exists a positive constant $C$ such that 
$$b(x) \,h(u) \leq C|x|^{\sigma} L_b(|x|) \,u^{p-1}\quad \text{for all }0<|x|\leq 1/2. $$
Similar to Step~2 of \cite[Lemma~5.2]{florica}, we construct a positive radial solution $v_\infty$ of 
\neweq{pla} -{\rm div}\,(\mathcal A(|x|)\, |\nabla v|^{p-2} \nabla v)+C|x|^{\sigma} L_b(|x|)\, v^{p-1} =0\quad \text{for } 0<|x|<1/2
\endeq
such that $v_\infty(|x|)\leq u(x)$ for $0<|x|\leq 1/2$. By a contradiction argument and Lemma~\ref{cp}, we find that
the radial solution $v_\infty$ of \eqref{pla} has a non-negative limit at $0$. To conclude that $\theta>0$, it suffices to show that $\lim_{r\to 0^+} v_\infty(r)>0$. 
By assuming that $\lim_{r\to 0^+} v_\infty(r)=0$, we arrive at a contradiction as follows. We use the change of variable $z(s)=v_\infty(r)$ with $s=\Phi(r)$. Then, we have $\lim_{s\to \infty}z(s)= 0$. Moreover, $z$ is a positive solution of the ordinary differential equation
\neweq{odes}
\left|\frac{dz}{ds} \right|^{p-2} \frac{d^2 z}{ds^2}=C_1 r^{N-1+\sigma} L_b(r)\,
\left[z(s)\right]^{p-1} \left| \frac{dr}{ds}\right| \ \ \text{for } s\in (\Phi (1/2),\infty),
\endeq 
where $C_1$ denotes a positive constant. 
Since $z''(s)>0$, then $z'(s)$ is increasing on $(\Phi (1/2),\infty)$ with
$\lim_{s\to \infty} z'(s)=0$. Therefore, using \eqref{odes}, we find that
$$ z(s)=C_2 \int_s^\infty \left( \int_0^{\Phi^{-1}(t)} \xi^{N-1+\sigma} L_b(\xi) \left[z(\Phi(\xi)) \right]^{p-1}\,{\mathrm d\xi}\right)^{\frac{1}{p-1}}\,dt \quad 
\text{for } s>\Phi (1/2),
$$
 where $C_2$ is a positive constant. Since $z$ is decreasing, we infer that 
\neweq{k:1} 
1/C_2\leq   \int_s^\infty
 \left( \int_0^{\Phi^{-1}(t)} \xi^{N-1+\sigma} L_b(\xi) \,{\mathrm d\xi}\right)^{\frac{1}{p-1}}\,dt \quad 
\text{for every } s>\Phi (1/2).
\endeq
Let $V(s)$ denote the right-hand side of \eqref{k:1}. 
We claim that $V(s)$ is well-defined and $V(s)\to 0$ as $s\to \infty$. Indeed, we have  
$\Phi\in RV_{-m_2}(0+)$ and thus $\Phi^{-1}\in RV_{-1/m_2}(\infty)$. Note that 
$r\longmapsto \int_0^r \xi^{N-1+\sigma}L_b(\xi)\,d\xi$ is regularly varying at $0^+$ with positive index given by $\sigma+N$. Consequently,  
$V$ is regularly varying at $\infty$ with {\em negative} index $(p+\sigma-\vartheta)/(p-N-\vartheta)$ so that the claim follows. 
Then, \eqref{k:1} leads to a contradiction, which proves that $\lim_{r\to 0^+} v_\infty(r)>0$ and, hence, $\theta>0$. 

To obtain that
$\lim_{|x|\to 0}u(x)=\theta$, $\lim_{|x|\to 0} |x||\nabla u(x)|=0$ and \eqref{umn} holds for all $\phi\in C^1_c(B_1)$, we proceed as in the special case of \cite[Proposition~3]{BCCT}. Since the ideas are very similar, we skip the details. \end{proof}

\begin{proof}[Proof of Lemma~\ref{radsy}] 
We show that $b(x)\, h(\Phi)\in L^1(B_{1/2})$ is a necessary condition for the existence of a positive solution of 
\eqref{radio1} with a weak or strong singularity at $0$. 
Let $v$ be a positive solution of \eqref{radio1} with $\lim_{r\to 0^+} v(r)/\Phi(r)=\lambda \neq 0$.

\vspace{0.2cm}
First, we consider the case $\lambda \in (0,\infty)$.   
Let $\Phi^{-1}(t)$ denote the inverse of $\Phi$, which exists for any $t>0$. 
By the change of variable $y(s)=v(r)$ with $s=\Phi(r)$, we find \eqref{rado2}. 
Since $v(r)\sim \lambda \Phi(r)$ as $r\to 0^+$, we have $y(s)\sim \lambda s$ as $s\to \infty$. Using that $d^2 y/ds^2\geq 0$, we get that 
$d y/ds$ is increasing on $(0,\infty)$ so that $\lim_{s\to \infty} d y/ds=\lambda$. 
We define $\Lambda$ by 
\begin{equation} \label{lambs} 
\Lambda(s):= \frac{C_{N,p}^{-p+1}}{p-1} [\Phi^{-1}(s)]^{N-1+\sigma} L_b(\Phi^{-1}(s))\,
L_h(s) \,s^{p-2} \left|\frac{dr}{ds} \right| \ \ \text{for }s>0 \text{ large}. 
\end{equation} 
Since $L_h \in RV_0(\infty)$ and  $y(s) \sim \lambda s$ as $s \to \infty$, we have $L_h(y(s))\sim L_h(s)$ as $s\to \infty$. 
We apply \eqref{doiii} to \eqref{rado2} to get that
\begin{equation} \label{sop}
\left\{ \begin{aligned}
& \frac{d^2 y}{ds^2}\sim \Lambda(s) [y(s)]^{q-p+2}\ \ \text{as } s\to \infty ,\\
& y'(s) \to \lambda \ \ \text{as } s\to \infty. 
\end{aligned}\right.
\end{equation}
By Taliaferro \cite[p.~96]{tal}, we get that $\int^\infty  t^{q-p+2}\Lambda(t)\,dt<\infty$. Then applying a change of variable $r=\Phi^{-1}(t)$ and using Remark~\ref{remark2}, we obtain that $b(x)\, h(\Phi)\in L^1(B_{1/2})$. 

\vspace{0.2cm}
Secondly, let $\lambda=\infty$. We adapt ideas from the proof of \cite[Lemma~5.8]{florica}. 
Choose $m\in (p-1,q_*)$ and for $t>0$,
 set $\chi(t)=t^{q_*-m} L_h(t)$. By the property in \eqref{nor} in the Appendix~\ref{appe}, we have 
$\lim_{t\to \infty} t\chi'(t)/\chi(t) =q_*-m>0$ and, hence, $\chi(t)$ is increasing for $t>0$ sufficiently large.  
Since $\lim_{r\to 0^+} v_*(r)/\Phi(r)=\infty$, there exists a constant $a_0>0$ such that $v_*(r)\geq a_0 \Phi(r)$ for all $0<r\leq 1/2$.
Then there exists a constant $c>0$ such that 
\neweq{lllhhh}
L_{h}(v_*) \, v_*^{q_*}\geq c \chi(\Phi(r)) \, v_*^m\quad \text{for all } r\in (0,1/2].
\endeq
Define a function $\tilde b (r):= c\,r^\sigma L_b(r) \, \chi(\Phi(r))$ for $r\in (0,1/2]$. We construct a positive radial solution $v_\infty$ of 
\neweq{ent} 
-{\rm div}\,(\mathcal A(|x|)\, |\nabla v|^{p-2} \nabla v)+\tilde b (|x|) \,v^m=0\quad \text{in } B^{*}_{1/2} 
\endeq
such that 
$v_*\leq v_\infty$ in $B^*_{1/2}$. Then, $v_\infty$ has a strong singularity at $0^+$. Since $\chi\in RV_{q_*-m}(\infty)$, we find that 
$\tilde b \in RV_{\tilde{\sigma}}(0+)$ with $\tilde{\sigma}$ given by $m(N+\sigma)/q_*-N$, which is greater than $\vartheta-p$ from our choice of $m$. We note that \eqref{ent} corresponds to \eqref{radio1} in the {\em critical} case with $r^\sigma L_b(r) = \tilde{b}(r)$, $L_h \equiv 1$ and $q=m$, where \eqref{doii} holds.
Using Remark~\ref{remark6} on \eqref{ent},  and the definition of $\tilde b$, we conclude that $b(x)\, h(\Phi)\in L^1(B_{1/2})$. This completes the proof of Lemma~\ref{radsy}.
\end{proof}

\section{Basic tools} \label{basic}

Throughout this section, let Assumptions $({\mathbf A_1})$--$({\mathbf A_3})$ hold.
Our aim is to prove the basic tools used in this paper: \emph{a priori} estimates 
(Lemma~\ref{apr}), a spherical Harnack-type inequality (Lemma~\ref{lemharnack}) and a regularity result (Lemma~\ref{reg}).

\begin{lemma}[\textit{A priori} estimates] \label{apr} For any $r_0\in (0,1/2)$, there exists a positive constant 
$C$, depending on $r_0$, such that \eqref{bst} holds for every
positive (sub-)solution of \eqref{une1}. 
\end{lemma}

\begin{proof} Fix $x_0\in \RR^N$ with $0<|x_0|\leq r_0$. We denote $\rho:=|x_0|/2$ and $p':=p/(p-1)$. Let 
\neweq{zf} \zeta(r):=r^{\frac{\sigma-\vartheta+p}{p}} \left[\frac{L_b(r)}{L_\A(r)}\right]^{\frac{1}{p}}\ \text{ for } r\in (0,r_0] \ 
\text{ and } f(t):=\frac{t^{1-\frac{q+1}{p}} [L_h(t)]^{-\frac{1}{p}}}{\int_t^\infty \xi^{-\frac{q+1}{p}} [L_h(\xi)]^{-\frac{1}{p}} \,d\xi}\quad\text{for }t>0\ \
\text{ large}. 
\endeq Let $c>0$ be a positive constant.  
We define   
$S=S_{x_0}:B_{\rho} (x_0)\to \RR$ by 
\begin{equation} \label{sde}
\int_{S(x)}^\infty t^{-\frac{q+1}{p}}\left[L_h(t)\right]^{-\frac{1}{p}}\,dt   =c \zeta(|x_0|)
\left[
1-\left(\frac{|x-x_0|}{\rho}\right)^{p'}\right] \quad \text{for every } x\in B_\rho(x_0).
\end{equation}

{\bf Claim:} {\em There exists a small positive constant $c$ depending on $r_0$, but independent of $x_0$ 
such that 
the function $S$ defined by \eqref{sde} is a super-solution of \eqref{une1} in $B_\rho(x_0)$, namely for $h_1$ as in Remark~\ref{asb}, it holds} 
\begin{equation} \label{super}
{\rm div}\,(\mathcal A(|x|)\, |\nabla S|^{p-2} \nabla S)  \leq b(x)\, h_1(S) \quad \text{in } B_\rho(x_0).
\end{equation}

Suppose the claim holds.
Since $S(x)\to \infty$ as $|x-x_0|\to \rho$, 
by the comparison principle of Lemma~\ref{cp}, we find that 
$u \leq S$ in $ B_{\rho} (x_0)$. In particular, we have $u(x_0)\leq S(x_0)$.
Since $\zeta$ is regularly varying at $0^+$ with positive index
$(p+\sigma-\vartheta)/p$, we have $\lim_{r\to 0^+}\zeta(r)= 0$ so that $\sup_{0<r\leq r_0}  \zeta(r)<\infty$. 
Since the right-hand side of \eqref{sde} is bounded from above by 
$c \sup_{0<r\leq r_0}  \zeta(r)$, for every $M>0$ there exists a small positive constant $c$ (depending on $M$ and $r_0$)
such that $S\geq M$ in $B_\rho(x_0)$ for every $0<|x_0|\leq r_0$. 
Using \eqref{zf} and 
 \eqref{sde}, we find that 
 \neweq{ooo} [S(x_0)]^{q-p+1} \,L_h(S(x_0))= \left[c \zeta(|x_0|) f(S(x_0))\right]^{-p}. 
 \endeq
 We fix $M>0$ as large as needed. Let $h_1$ and $h_2$ be as in Remark~\ref{asb} of Appendix~\ref{appe}. We can thus assume that 
$h_2(t)\leq 2t^q L_h(t)$ for all $t\geq M$. 
By Karamata's Theorem in Appendix~\ref{appe}, we have $\lim_{t\to \infty} f(t)=(q-p+1)/p>0$. 
Since $u(x_0)\leq S(x_0)$, using \eqref{ooo} and \eqref{9.4}, we can find a positive constant $C_1=C_1(r_0)$ independent of $x_0$ such that
 \neweq{few}
  \frac{|x_0|^p b(x_0)}{\mathcal A(|x_0|)} \frac{h(u(x_0))}{[u(x_0)]^{p-1}}\leq 
    \frac{|x_0|^p b(x_0)}{\mathcal A(|x_0|)} \frac{h_2(S(x_0))}{[S(x_0)]^{p-1}}\leq \frac{2}{[c f(S(x_0))]^p}\,\frac{b(x_0)}{|x_0|^\sigma L_b(|x_0|)}\leq C_1.
 \endeq  
 Since \eqref{few} holds for every $0<|x_0|\leq r_0$, 
we conclude the assertion of Lemma~\ref{apr}. 

\vspace{0.2cm}
{\bf Proof of Claim.} 
By \eqref{sde}, we find that 
\neweq{ddif}
|\nabla S(x)|^{p-2}\nabla S(x)=\left(cp'\right)^{p-1} \rho^{-p} \left[\zeta(|x_0|)\right]^{p-1} 
\left[S^{q+1}(x) \,L_h(S(x))\right]^{\frac{1}{p'}} (x-x_0)\quad \text{in } B_\rho(x_0).
\endeq

Using $f$ given by \eqref{zf}, we denote by $T_{x_0}(x)$ the following quantity 
\neweq{tx0} \left(\frac{|x-x_0|}{\rho}\right)^{p'}\left( q+1+\frac{S(x) L_h'(S(x))}{L_h(S(x))}\right) + 
f(S(x))
 \left[
1-\left(\frac{|x-x_0|}{\rho}\right)^{p'}
\right] \left(N+\frac{|x| \mathcal A'(|x|)}{\mathcal A(|x|)} \, \frac{(x-x_0)\cdot x}{|x|^2}\right).
\endeq
With $T_{x_0}(x)$ given by \eqref{tx0}, we derive that 
\neweq{slo} 
 {\rm div}\,(\mathcal A(|x|)\, |\nabla S|^{p-2} \nabla S) =\left(p'\right)^{p-1} (2c)^{p} \left(\frac{|x|}{|x_0|}\right)^{\vartheta}\,
 \frac{L_\A(|x|)}{L_\A(|x_0|)}
  \,|x_0|^\sigma L_b(|x_0|)\, S^q L_h(S)\,T_{x_0}(x) . 
\endeq
By Assumption~$({\mathbf A_1})$ and Remark~\ref{asb}
in Appendix~\ref{appe}, we have $\lim_{r\to 0^+} r\A'(r)/\A(r)=\vartheta$ and $\lim_{t\to \infty} t L_h'(t)/L_h(t)=0$. 
Recall that $\lim_{t\to \infty} f(t)=(q-p+1)/p$. 
Moreover, by Proposition~\ref{pn} in Appendix~\ref{appe}, there exist positive constants $c_i$ $(0\leq i\leq 3)$ depending on $r_0$, but independent of $x_0$ such that 
 $$c_0 \,L_\A(|x_0|) \leq L_\A(|x|)\leq c_1\, L_\A(|x_0|)
 \ \ \text{and  }\ 
c_2 \,L_b(|x|)\leq  L_b(|x_0|)\leq c_3 \,L_b(|x|) 
 $$
 for every $x,x_0$ such that $0<|x_0|\leq r_0$ and $ |x|/|x_0|\in [1/2, 3/2]$. Thus, using \eqref{aoz} and \eqref{slo}, we conclude \eqref{super} 
by taking in \eqref{sde} a small constant $c>0$ depending on $r_0$, but independent of $x_0$. This completes the proof of Lemma~\ref{apr}. 
\end{proof}

\begin{lemma}[Harnack-type inequality] \label{lemharnack}
Fix $r_0 \in (0,1/2)$. There exists a positive constant $K$ (depending on $p$, $N$ and $r_0$)
such that 
for every positive solution $u$ of \eqref{une1}, we have \eqref{eqharnack}. 
\end{lemma}
\begin{proof}

We first observe that \eqref{une1} is equivalent to 
\begin{equation} \label{une1equiv}
-{\rm div}\,(|\nabla u|^{p-2} \nabla u)+ \frac{\A'(|x|)}{\A(|x|)} \,|\nabla u|^{p-2}\frac{\nabla u \cdot x}{|x|}+\frac{b(x)\,h(u)}{\A(|x|)\,u^{p-1}}\,u^{p-1}=0 \quad \text{in } B^*.
\end{equation}
Let $b_1$ and $b_2$ denote two non-negative functions as follows
\begin{equation} \label{bud1} b_1(x):=    \frac{|\A'(|x|)|}{\A(|x|)} \ \ \text{and} \ \ 
[b_2(x)]^p:=\frac{b(x)\,h(u)}{\A(|x|)\,u^{p-1}}\quad \text{for } 0<|x|\leq  r_0. 
\end{equation}
By \eqref{a3} and Lemma~\ref{apr}, there exists a positive constant $C_1$, depending on $r_0$, such that 
\begin{equation} \label{c1e} |x|\, b_1(x)\leq C_1\ \ \text{and}\ \ 
|x| \,b_2(x) \leq C_1\quad \text{for all } 0<|x|\leq r_0.  
\end{equation}
Fix $x_0 \in \RR^N$ such that $0<|x_0|\leq r_0/2$ and set $\rho:=|x_0|/2$. We use $\mu$ to denote $$\mu=\mu_{x_0}:= \max\{\|b_1\|_{{L^\infty}(B_\rho(x_0))}, \|b_2\|_{{L^\infty}(B_\rho(x_0))} \} .$$ 
Since $ \rho\leq |x|$ for every $x\in B_\rho(x_0)$, from \eqref{c1e} it follows that
\begin{equation} \label{rhom} \rho \mu \leq C_1\quad \text{for every } x\in B_\rho(x_0). \end{equation}
We apply the Harnack inequality of \cite[Theorem 1.1]{ntrud} for \eqref{une1equiv} on $B_{|x_0|/2}(x_0)$ where the structure conditions in (1.2) and (1.3) of \cite{ntrud} are satisfied with
$a_0=1$ and $a_i =b_0=b_3=0$ for $i\in \{1,2,3,4\}$. Hence, 
there exists a positive constant $k$, depending only on $p$, $N$ and $ \rho \mu$, such that
\begin{equation} \label{preharn}
\sup_{x \in B_{\rho/3}(x_0)} u(x) \leq k \inf_{x \in B_{\rho/3}(x_0)} u(x).
\end{equation}
By the covering argument in \cite{fav}, any two points $x_1$ and $x_2$ in $\RR^N$ such that $0<|x_1|=|x_2|\le r_0 /2$ can be joined by ten overlapping balls of radius $|x_1|/6$ with centres positioned on $\partial B_{|x_1|}(0)$. Thus, by \eqref{rhom} and \eqref{preharn}, we obtain \eqref{eqharnack} with $K=k ^{10}$, where $K$ is a positive constant depending on $p$, $N$ and $r_0$. 
\end{proof}

\begin{lemma}[A regularity result] \label{reg}
Fix $r_0\in (0,1/4)$ and $\delta \geq 0$. Let $g\in C(0,1)$ be a positive function such that  
$g$ is regularly varying at $0$ with index $-\delta$. 
Suppose that $u$ is a positive solution of \eqref{une1} and $C_0>0$ is a constant such that
\begin{equation} \label{tolone}
0<u(x)\le C_0\,g(|x|) \quad \text{for } 0<|x|<2r_0.
\end{equation}
Then there exist positive constants $C>0$ and $\alpha \in (0,1)$ such that
\begin{equation} \label{toltwo}
|\nabla u(x)|\le C\,\frac{g(|x|)}{|x|} \quad \text{and} \quad |\nabla u(x)-\nabla u(x')| \le C\, \frac{g(|x|)}{|x|^{1+\alpha}} |x-x'|^\alpha
\end{equation}
for any $x$, $x'$ in $\RR^N$ satisfying $0<|x| \le |x'|<r_0$.
\end{lemma}

\begin{proof}
We use an argument close to \cite[Lemma~4.1]{cd}, which is similar to \cite[Lemma~1.1]{fav} (see also \cite[Lemma~3]{BCCT}). 
There is, however, one essential difference with respect to the derivation of the first inequality in \eqref{toltwo}. 
We show below the main modifications compared with \cite[Lemma~4.1]{cd}. 

Using \eqref{une1equiv} and 
defining $\Psi_\beta$ as in (4.5) of \cite{cd}, that is 
$\Psi_\beta(\xi) := u(\beta \xi)/g(\beta)$ for $\xi \in \bar{\Gamma}$, where $\beta \in (0,r_0/6)$ is fixed,   
we see 
that $\Psi_\beta$ satisfies an equation of the form (4.3) of \cite{cd}, namely
\begin{equation} \label{eqtol}
- {\rm div}\,(|\nabla \Psi_\beta|^{p-2} \nabla \Psi_\beta)+B_\beta=0 \quad \text{in } \Gamma, \quad \text{where } \Gamma:=\{y\in \RR^N : 1<|y|<7\}.
\end{equation}
However, instead of (4.7) in \cite{cd}, the expression of $B_\beta$ is more complicated here, involving a gradient term, namely 
\begin{equation} \label{bbe} B_\beta(\xi):= \frac{\beta^p}{[g(\beta)]^{p-1}}\, b(\beta \xi) \frac{h(u(\beta \xi))}{\A(\beta|\xi|)}-
\frac{\beta \A'(\beta|\xi|)}{\A(\beta |\xi|)} \,|\nabla \Psi_\beta|^{p-2}\,\frac{\nabla \Psi_\beta(\xi) \cdot \xi}{|\xi|}\ \ \text{for }\xi\in \Gamma. 
\end{equation}   

{\sc Claim:} \emph{The functions $\Psi_\beta$ and $B_\beta$ are in $L^\infty (\Gamma)$ with their $L^\infty$-norms bounded above by 
a positive constant independent of $\beta\in (0,r_0/6)$.}

\vspace{0.2cm}
{\em Proof of claim.}  
For $\Psi_\beta$, we can proceed exactly as in \cite{cd}. We thus need to prove the claim only for $B_\beta$. 
Using Lemma~\ref{apr} and \eqref{tolone}, jointly with (4.10) in \cite{cd}, we find that the $L^\infty(\Gamma)$-norm of the 
\emph{first term} in the right-hand side of \eqref{bbe} is bounded above by a constant independent of $\beta$.  

Assume for now that the first inequality in \eqref{toltwo} is proved. Then we can infer that 
$|\nabla \Psi_\beta(\xi)|\leq C g(\beta |\xi|)/g(\beta)$ for every $\xi\in \Gamma$. 
Hence, using \eqref{c1e}, as well as (4.10) in \cite{cd}, we could conclude the claim for $B_\beta$ given by \eqref{bbe}.  

\vspace{0.2cm}
Since $B \in L^\infty (\Gamma)$ and $\Psi \in L^\infty (\Gamma) \cap W^{1,p}(\Gamma)$ is a weak solution of \eqref{eqtol},
from the $C^{1,\alpha}$-regularity result of Tolksdorf \cite{tol}, we conclude that 
there exist constants $\alpha = \alpha(N,p) \in (0,1)$ and $\tilde{C}=\tilde{C}\left(N,p,\|\Psi\|_{L^\infty (\Gamma)}, \|B\|_{L^\infty (\Gamma)}\right) >0$ such that
\begin{equation} \label{tolresult}
\|\nabla \Psi \|_{C^{0,\alpha}(\Gamma^*)} \le \tilde{C}, \quad \text{where } \Gamma^* := \{y \in \RR^N : 2<|y|<6 \}.
\end{equation}
This fact is then used to derive the second inequality in \eqref{bbe} (see \cite{cd} for details).

\vspace{0.2cm}
{\em Proof of the first inequality in \eqref{toltwo}.} Our proof here is 
 different from both \cite[Lemma~4.1]{cd} and \cite[Lemma~3]{BCCT}. 
We require a new argument to that of \cite{cd} 
as we used the first inequality in \eqref{toltwo} to derive \eqref{tolresult}. The ideas in \cite{BCCT} work for the special 
case $p=2$. In our general situation, we apply Theorem~1 in Tolksdorf \cite{tol} for the function
$v$ in \eqref{res}.       
More precisely, let $x_0\in \RR^N$ be fixed such that $0<|x_0|\leq r_0$ and set $\rho:=|x_0|/2$. 
We define 
$v=v_{x_0}:B_1\to (0,\infty)$ by
\begin{equation} \label{res}
v(y):=\frac{u(x_0+\rho y)}{g(|x_0|)}\quad \text{for every } y\in B_1.
\end{equation}
Since $u$ satisfies \eqref{une1equiv}, by using the formula for $\nabla v$ derived from \eqref{res}, that is  
\begin{equation} \label{view} \nabla v(y)=\frac{\rho}{g(|x_0|)} (\nabla u)(x_0+\rho y) \quad \text{for } y\in B_1,
\end{equation} we obtain that $v$ is a positive solution of the following equation
$$ -{\rm div}\,( |\nabla v|^{p-2} \nabla v) +\tilde B(y,v,\nabla v) =0 \quad \text{in } B_1,
$$
where we define $\tilde B(y,v,\nabla v)$ to be
$$  \tilde B(y,v,\nabla v)=-\frac{\rho \,\A'(|x_0+\rho y|)}{\A(|x_0+\rho y|)} |\nabla v|^{p-2} \frac{\nabla v(y)
\cdot (x_0+\rho y)}{|x_0+\rho y|}+\rho^{\,p} \frac{b(x_0+\rho y)\,h(v)}{\A(|x_0+\rho y|)\,v^{p-1}}\,v^{p-1}.
$$
Since $|x_0+\rho y|\in [\rho,3\rho]$ for all $y\in B_1$, in view of \eqref{a3} and \eqref{c1e}, we find that
\begin{equation} \label{bun} |\tilde B(y,v,\nabla v)| \leq A_1 |\nabla v|^{p-1} +A_2\, v^{p-1}  
\end{equation} for some positive constants $A_1$ and $A_2$, which depend on $r_0$, but are independent of $x_0$. 
Using the assumptions on $g$, namely $g$ is regularly varying at $0$, we obtain (similar to (4.10) in \cite{cd}) that 
$$ \underline c\, g(|x_0|) \leq  g(|x_0+\rho y|)  \leq \overline c \,g(|x_0|) \quad \text{for all } y\in B_1,$$
where $\underline c$ and $\overline c$ are positive constants, which depend on 
$r_0$, but are independent of $x_0$ satisfying $0<|x_0|< r_0$. 
Moreover, from \eqref{tolone} and \eqref{res}, we deduce that 
$$ v(y)\leq \overline{c} \,C_0 \quad \text{for every } y\in B_1. $$
Thus, in view of \eqref{bun}, we can find a positive constant $A_3=A_3(r_0)$, which is independent of $x_0$ such that 
$$ |\tilde B(y,v, \eta)| \leq A_3 (1+|\eta|)^p \quad \text{for all } y\in B_1\ \text{and } \eta\in \RR^N. 
$$
Hence, we can apply Theorem~1 in Tolksdorf \cite{tol} to obtain a constant $A_4$, which 
depends on $N$, $p$ and $A_3$, but is independent
of $x_0$, such that $|\nabla v(0)|\leq A_4$. This, jointly with \eqref{view}, proves that 
$$ |\nabla u(x_0)|\leq 2 A_4 \,\frac{g(|x_0|)}{|x_0|}\quad \text{for every } 0<|x_0|<r.$$ 
This completes the proof of Lemma~\ref{reg}. 
\end{proof}

\section{Proof of Theorem~\ref{wsol}: Existence and uniqueness} \label{weaksol}
Let Assumptions $({\mathbf A_1})$--$({\mathbf A_3})$ hold. Let $h$ be non-decreasing on $[0,\infty)$ and $g \in C^1 (\partial B_1)$ be a non-negative function. We study the existence of solutions for the following problem
\neweq{probe2} 
\left\{ \begin{aligned}
& {\rm div}\,(\mathcal A(|x|)\, |\nabla u|^{p-2} \nabla u) =b(x) \,h(u)\quad \text{in } B^*:=B_1\setminus\{0\},\\
&  \lim_{|x|\rightarrow 0} \frac{u(x)}{\Phi(x)}=\lambda, \quad u\big|_{\partial B_1}=g, \quad  u>0 \quad \text{in } B^*. 
\end{aligned}
\right.
\endeq

We treat separately the following cases: $\lambda=0$, $\lambda\in (0,\infty)$ and $\lambda=\infty$. 
For the construction of a solution of \eqref{probe2}, 
we adapt ideas from \cite[Theorem~1.2]{cd} (where $\A=1$), see also 
\cite[Proposition~5]{BCCT}, where $p=2$, $b=1$ and $h(t)=t^q$.  
We denote $C_0:=\max_{|x|=1} g(x)$. For every $n\geq 2$ and $0\leq \lambda<\infty$, we consider the auxiliary problem 
\begin{align} \label{ham}
 \begin{cases}
  &{\rm div}\,(\mathcal A(|x|)\, |\nabla u|^{p-2} \nabla u)= b(x)\,h(u) \quad \text{in } D_n:=B_1 \setminus \overline{B_{1/n}}, \\
  &u(x) =  \lambda\, \Phi(|x|) + C_0 \quad \text{for } |x|=1/n,\\
  & u \, \big|_{\partial B_1}=g. 
 \end{cases}
\end{align}
For $\lambda=0$, we further assume that $g \nequiv 0$ on $\partial B_1$. 
By the method of sub-super-solutions and Lemma~\ref{cp}, the problem \eqref{ham} admits a unique non-negative solution $u_{n,\lambda,g}$, which is continuous on $\overline  {D_n}$. For simplicity, whenever $\lambda$ and $g$ are fixed, we simply write $u_n$ instead of $u_{n,\lambda,g}$.    
By the strong maximum principle (see Theorem~2.5.1 of \cite{PS}), we see that $u_n$   
positive in $D_n$.  
Moreover, by Lemma~\ref{cp}, we infer that 
\neweq{mmi}0< u_{n+1}\leq u_n\leq \lambda \Phi(|x|)+ C_0 \quad \text{in  } D_n.\endeq
By Lemma~\ref{reg}, we have that, up to a subsequence,  
$u_n\to u_{\lambda,g}$ in $C^{1}_{\rm loc} (B^*)$ and, moreover,  for some $\alpha\in (0,1)$, we find that
$u_{\lambda,g}$  is a non-negative $C^{1,\alpha}_{\rm loc}(B^*)\cap C(\overline{B_1}\setminus\{0\})$-solution of the problem
\neweq{probe} 
\left\{ \begin{aligned}
& {\rm div}\,(\mathcal A(|x|)\, |\nabla u|^{p-2} \nabla u) =b(x) \,h(u)\quad \text{in } B^*:=B_1\setminus\{0\},\\
&u\big|_{\partial B_1}=g. 
\end{aligned}
\right.
\endeq
By the strong maximum principle, $u_{\lambda,g}$ is positive in $B^*$ (using here that $g\not\equiv 0$ on $\partial B_1$ when $\lambda=0$).   
From \eqref{mmi}, we find that $\limsup_{|x|\to 0} u_{\lambda,g}(x)/\Phi(|x|)\leq \lambda$. In particular, the problem  \eqref{probe2} with $\lambda=0$  
admits $u_{\lambda,g} $ as a solution. 

\vspace{0.2cm} 
{\em Proof of Theorem~\ref{wsol}(i).} It remains to show the uniqueness of the 
solution of \eqref{probe2} with $\lambda=0$. 
Let $u_1$ and $u_2$ be two solutions of \eqref{probe2} with $\lambda=0$. To show that $u_1=u_2$ in $B^*$, we proceed as in Proposition~4 in \cite{BCCT} 
with modifications appearing here due to our more general setting.   
By Lemma~\ref{remcon}, $u_1$ and $u_2$ can be extended by continuity at $0$. 
Since $u_1$, $u_2 \in C^1(B^*) \cap C(\overline{B_1})$ with $u_1=u_2=g$ on $\partial B_1$, then $u_1=u_2$ in $B_1$ would be a consequence of the following claim.

\vspace{0.2cm}
{\em Claim:} We have $\nabla (u_1-u_2)(x)=0$ for all $x \in B^*$.
\vspace{0.2cm}

{\em Proof of Claim.}  Assume by contradiction that there exists $x_0 \in B^*$ such that $|\nabla (u_1-u_2)(x_0)|>0$. We fix $r_0$ small such that  $0<r_0< \min\{1-|x_0|,|x_0|\}$, which ensures that $\overline{B_{r_0}(x_0)} \subset B^*$. Since $u_1-u_2 \in C^1(B^*)$, by making $r_0$ smaller if necessary, we can assume that $|\nabla (u_1-u_2)(x)|>0$ on $\overline{B_{r_0}(x_0)}$ and thus $|\nabla u_1(x)|+|\nabla u_2(x)|>0$ on $\overline{B_{r_0}(x_0)}$. Hence, there exists a positive constant $c_0$ such that
\begin{equation} \label{anineq}
 (|\nabla u_1(x)| + |\nabla u_2(x)|)^{p-2} \, |\nabla (u_1-u_2)(x)|^2\ge c_0 \quad \text{for all } x \in \overline{B_{r_0}(x_0)}.
\end{equation}
By Proposition 17.3 in \cite[p. 235]{chipot}, we know that there exists a positive constant $c_p$ such that
\begin{equation} \label{chipeq}
(|\xi|^{p-2}\xi - |\eta|^{p-2}\eta)\cdot (\xi - \eta) \ge c_p\left(|\xi|+|\eta|\right)^{p-2}|\xi-\eta|^2 \quad \text{for every } \xi, \eta \in \RR^N. 
\end{equation}
Thus using \eqref{anineq} and \eqref{chipeq}, we find for all $x \in \overline{B_{r_0}(x_0)}$ that
\begin{equation} \label{defh} 
\mathcal{H} (x):= (|\nabla u_1(x)|^{p-2} \nabla u_1 (x) - |\nabla u_2 (x)|^{p-2} \nabla u_2(x)) \cdot \nabla (u_1-u_2)(x) \ge c_p c_0.
\end{equation}
For any $\varepsilon\in (0,1/2)$, we denote $D_\varepsilon:=B_1 \setminus \overline{B_\varepsilon}$. 
Let $w_\varepsilon$ be a non-decreasing and smooth function on $(0,\infty)$ such that 
\neweq{def-w}\left\{ \begin{aligned}
& w_\varepsilon(r)\in (0,1)&& \text{if } \varepsilon<r<2\varepsilon ,&\\ 
& w_{\varepsilon}(r)=1 && \text{if } r\geq 2\varepsilon,&\\
& w_\varepsilon(r)=0 &&\text{if } 0<r\leq \varepsilon.&
\end{aligned}\right. \endeq 

We choose $\varepsilon>0$ small such that $2\varepsilon<|x_0|-r_0$, which yields that
$ \overline{B_{r_0}(x_0)} \subseteq D_{2\varepsilon} \subset D_{\varepsilon}$. Since $w_\varepsilon(|x|)=1$ for all $x\in D_{2\varepsilon}$, by using \eqref{defh}, we arrive at
\neweq{inf} 
\int_{D_{\varepsilon}} w_\varepsilon (|x|)\, \A(|x|) \, \mathcal H(x)\, dx\geq 
\int_{B_{r_0}(x_0)} \A(|x|)\,\mathcal H(x)\,dx\geq c_p \,c_0\,\omega_N r_0^N\, \min_{x\in \overline{B_{r_0}(x_0)}} \A(|x|):=c_{p,\A}.
\endeq
Since $\A\in C(0,1]$ is a positive function and $\overline{B_{r_0}(x_0)} \subset B^*$, we then obtain that $c_{p,\A}$ is a positive constant. 

Observe that $u_1$, $u_2$ and $w_\varepsilon$ belong to $W^{1,p}(D_{\varepsilon})\cap 
L^\infty (D_{\varepsilon})$. We define $\phi_\varepsilon(x):=(u_1-u_2)(x)\,w_\varepsilon(|x|)$ for all $x\in B^*$. Since 
$\phi_\varepsilon|_{\partial D_{\varepsilon}}=0$, it follows by the product rule that $\phi_\varepsilon \in W_0^{1,p}(D_{\varepsilon})$. Using the density of $C_c^1(D_{\varepsilon})$  in 
$W_0^{1,p}(D_\varepsilon)$, we have  
\neweq{dis} 
\int_{D_\varepsilon} \mathcal A(|x|)\, |\nabla u_j|^{p-2} \nabla u_j\cdot \nabla \phi_\varepsilon\,dx +\int_{D_\varepsilon} 
b(x)\,h(u_j)\,\phi_\varepsilon\,dx=0 \ \text{with } j=1,2.
\endeq
In particular, by subtracting the relation in \eqref{dis} with $j=2$ from the one corresponding to $j=1$, we obtain that
\neweq{difs}
\int_{D_\varepsilon} w_\varepsilon (|x|)\, \A(|x|) \, \mathcal H(x)\, dx + \int_{D_\varepsilon} b(x)\left(h(u_1)-h(u_2)\right)\left(u_1-u_2\right) w_\varepsilon(|x|) \, dx =-K_\varepsilon, \endeq
where $\mathcal H$ is given by \eqref{defh} and 
$K_\varepsilon$ is defined by 
\neweq{def-K} K_\varepsilon=  
\int_{\varepsilon<|x|<2\varepsilon} |x|^\vartheta\,  L_\A(|x|)\, w_\varepsilon ' (|x|) \left(u_1-u_2\right) \left(|\nabla u_1|^{p-2} \nabla u_1 - |\nabla u_2|^{p-2} \nabla u_2\right) \cdot \frac{x}{|x|} \, dx. \endeq
Since $w_\varepsilon(2\varepsilon)=1$ and $w_\varepsilon(\varepsilon)=0$ (see \eqref{def-w}), we observe that 
$$ \mathcal L_\varepsilon:=\int_{\varepsilon<|x|<2\varepsilon}  |x|^{\vartheta-p+1}  L_\A(|x|)\, w_\varepsilon' (|x|) \,dx
=|\partial B_1| \int_{\varepsilon}^{2\varepsilon}  
r^{\vartheta+N-p}  L_\A(r)\, w_\varepsilon' (r) \,dr\leq |\partial B_1| \max_{r\in [\varepsilon,2\varepsilon]}
\{r^{\vartheta+N-p}\,L_\A(r)\}. $$
Using that $\vartheta+N-p>0$ and $L_\A$ is slowly varying at zero, we get that $\lim_{r\to 0^+} r^{\vartheta+N-p}  L_\A(r)=0$.
(In relation to Remark~\ref{equal}, we note that if $\vartheta+N-p=0$ and $\limsup_{r\to 0^+} L_\A(r)<\infty$, then we get that 
$\limsup_{\varepsilon\to 0^+} \mathcal L_\varepsilon \in (0,\infty)$.)
Thus, using \eqref{def-K}, jointly with $|x||\nabla u_j|\to 0$ as $|x|\to 0$ for $j=1,2$ (see Lemma~\ref{remcon}), we find that 
$$   |K_\varepsilon| \leq \left(\|u_1\|_{L^\infty(B_1)} +\|u_2\|_{L^\infty(B_1)} \right)\, \mathcal L_\varepsilon\,
 \max_{\varepsilon\leq |x|\leq 2\varepsilon}  |x|^{p-1} \left( |\nabla u_1(x)|^{p-1}+|\nabla u_2(x)|^{p-1}\right)
 \to 0\ 
\text{as } \varepsilon\to 0^+.$$
Hence, we can fix $\varepsilon>0$ small enough to ensure that $|K_\varepsilon|<c_{p,\A}$, where $c_{p,\A}$ is the positive constant
appearing in \eqref{inf}. 
Since the second term in the left-hand side of \eqref{difs} is non-negative, from \eqref{inf} and \eqref{difs}, we get a contradiction.  
This proves the claim, which concludes the proof of of the uniqueness of the solution of \eqref{probe2} with $\lambda=0$.  

\vspace{0.2cm} 
{\em Proof of Theorem~\ref{wsol}(ii).} 
If \eqref{probe2} has a solution for $\lambda\in (0,\infty]$, then 
$b(x)\,h(\Phi) \in L^1(B_{1/2})$ from Theorem~\ref{th1.2}(b). 

\vspace{0.2cm}
{\em Claim 1:} 
If $b(x)\,h(\Phi) \in L^1(B_{1/2})$, then $u_{\lambda,g} $ constructed above for $\lambda\in (0,\infty)$ is a solution of \eqref{probe2}. 

\vspace{0.2cm}
{\em Proof of Claim 1.}
We need only show that 
$\liminf_{|x|\to 0} u_{\lambda,g}(x)/\Phi(|x|)\geq \lambda$. We note that \eqref{phicondition} is equivalent to $\int^\infty  t^{q-p+2}\Lambda(t)\,dt<\infty$, where 
$\Lambda$ is defined by \eqref{lambs}.
Then, by \cite[Theorem 2.4]{tal}, if $R>0$ is large, there exists a positive proper solution of 
the following problem  
\begin{equation} \label{sops}
\left\{ \begin{aligned}
& \frac{d^2 y}{ds^2}= \Lambda(s) [y(s)]^{q-p+2}\ \ \text{for } s\in (R, \infty) ,\\
& y'(s) \to \lambda \ \ \text{as } s\to \infty\ \ \text{and } y(R)\in (0,\infty).  
\end{aligned}\right.
\end{equation}
Using the transformation $w(r)=y(s)$ with $r=\Phi^{-1}(s)$ and Remark~\ref{asb}, we obtain that 
\begin{equation} \label{soi}
\left\{ \begin{aligned}
& {\rm div}\,(\mathcal A(|x|)\, |\nabla w|^{p-2} \nabla w) \sim b(x)\, h_2(w(|x|)) \ \ \text{as } |x|\to 0^+ ,\\
& w(r) \sim \lambda \Phi(r)  \ \ \text{as } r\to 0^+.
\end{aligned}\right.
\end{equation}
Hence, for every $\varepsilon\in (0,1)$, there exists $r_\varepsilon\in (0,\Phi^{-1}(R))$ such that 
$(1-\varepsilon) \,w$ is a sub-solution of 
\begin{equation} \label{rov}
{\rm div}\,(\mathcal A(|x|)\, |\nabla v|^{p-2} \nabla v)=b(x)\, h_2(v) \quad \text{in } B_{r_\varepsilon}^*:=B_{r_\varepsilon}\setminus\{0\}.
\end{equation}
Recall that $u_{n,\lambda,g}$, in short $u_n$, represents the unique non-negative solution of \eqref{ham}.   
Since $w(r) \sim \lambda \Phi(r)$ as $r\to 0^+$ (see \eqref{soi}), there exists $n_\varepsilon\geq 1$ large such that
$$ (1-\varepsilon) \,w(1/n)\leq \lambda \,\Phi(1/n)\leq u_n(x)\quad \text{for every } |x|=1/n   \ \text{and all } n\geq n_\varepsilon.
$$
Let $C_\varepsilon:=\max_{r=r_\varepsilon} w(r)$. Since $u_n$ is a positive super-solution of \eqref{rov} due to our choice of $h_2$, by Lemma~\ref{cp}, we 
have
$$ (1-\varepsilon)\, w\leq u_n+C_\varepsilon\quad \text{for } 1/n<|x|<r_\varepsilon \ \ \text{and all } n\geq n_\varepsilon. 
$$
By letting $n\to \infty$, we find that $(1-\varepsilon)\,w\leq u_{\lambda,g} +C_\varepsilon$ in $B_{r_\varepsilon}^*$. Hence,
we conclude that $\liminf_{|x|\to 0} u_{\lambda,g}(x)/\Phi(|x|)\geq (1-\varepsilon)\lambda$. Since $\varepsilon\in (0,1)$ is arbitrary, we obtain that
$\liminf_{|x|\to 0} u_{\lambda,g}(x)/\Phi(|x|)\geq \lambda$. Since $\limsup_{|x|\to 0} u_{\lambda,g}(x)/\Phi(|x|)\leq \lambda$, it follows that
$u_{\lambda,g}$ is a solution of \eqref{probe2} for $\lambda\in (0,\infty)$. 

\vspace{0.2cm}
{\em Claim 2:} If $b(x)\,h(\Phi) \in L^1(B_{1/2})$, then there exists a solution of \eqref{probe2} with $\lambda=\infty$. 

\vspace{0.2cm}
{\em Proof of Claim 2.} 
Let $k$ be any positive integer and denote by $u_{k,g}$ the solution we constructed earlier for \eqref{probe2} with $\lambda$ replaced by $k$. Then, by the comparison principle
(Lemma~\ref{cp}), we find that $0<u_{k,g}\leq u_{k+1,g}$ in $B^*$. We show that for every fixed $x\in B_1\setminus\{0\}$, 
there exists $\lim_{k\to \infty} u_{k,g} (x)\in (0,\infty)$. Indeed, since $|x|>0$, we can fix $\rho=\rho_x$ such that $0<\rho<\min\{|x|,1/4\}$. 
Hence, by Lemma~\ref{apr}, there exists $C_{\rho}>0$ such that  
$u_{k,g}(y)\leq C_\rho$ for all $|y|=\rho$ and every $k\geq 1$. By Lemma~\ref{cp}, it follows that $u_{k,g}(y)\leq \max\{C_0,C_\rho\}$ for all
$\rho\leq |y|\leq 1$ and all $k\geq 1$, where $C_0=\max_{|x|=1} g(x)$. Hence, for all $x\in \overline{B_1}\setminus\{0\}$, we can define
$u_{\infty,g}(x):=\lim_{k\to \infty} u_{k,g}(x)$. Moreover, by Lemma~\ref{reg}, we have that, up to a subsequence, 
$u_{k,g}\to u_{\infty,g}$ in $C^1_{\rm loc}(B^*)$ and $u_{\infty,g}$ is a solution of \eqref{probe2} with $\lambda=\infty$. 
This concludes Claim 2 and the proof of Theorem~\ref{wsol}(ii).
 
\vspace{0.2cm} 
{\em Proof of Theorem~\ref{wsol}(iii).}  
Assume that  $b(x)\,h(\Phi) \in L^1(B_{1/2})$ and $h(t)/t^{p-1}$ is non-decreasing for $t>0$. We show the uniqueness of the solution of \eqref{probe2} in any of the following situations:

\vspace{0.2cm}
(A) $\lambda\in (0,\infty)$; 

(B) $\lambda=\infty$ and $q<q_*$; 

(C) $\lambda=\infty$ and $q=q_*$, assuming also that either \eqref{doii} or \eqref{doi} holds.

\vspace{0.2cm}
Indeed, if $u_1$ and $u_2$ are arbitrary solutions  of \eqref{probe2} corresponding to the same $\lambda$ and $g$, then $\lim_{|x|\to 0} u_1(x)/u_2(x)=1$. This is evident
in Case (A), while for the Cases (B) and (C), we use Theorem~\ref{th1.2}(a) to obtain the same asymptotic behaviour near zero for any positive solution of \eqref{une1} with a strong singularity at $0$. 
The uniqueness claim follows from Lemma~\ref{cp} as in the proof of
 \cite[Theorem~1.2]{cd}. This completes the proof of Theorem~\ref{wsol}. \qed

\appendix
\renewcommand\thesection{\Alph{section}}
\section{Regular variation theory} \label{appe}

The regular variation theory initiated by Karamata in the 1930's has been very 
fruitful in statistics in connection with extreme value theory (statistical estimation 
of tails, rates of convergence). It also plays a crucial role in probability theory 
(weak limit theorems such as central limit theorem and the weak law of large numbers; 
branching processes; stability and domains of attraction; fluctuation theory; renewal 
theory). The applications are much broader, including areas such as analytic number 
theory, financial engineering and complex analysis (see \cite{BGT} for a comprehensive 
treatment of regular variation theory and its applications). 

We recall below the concepts and properties of regularly varying functions needed in this paper, see  \cite{BGT,resn,se}. 

\begin{e-definition}[Regularly varying functions] \label{def-reg} \ \
\begin{enumerate} 
\item[{\rm (a)}] A positive measurable function $L$ defined on a neighbourhood of $\infty$ is called \emph{slowly varying}
at $\infty$ if 
$$ \lim_{t\to \infty} \frac{L(\xi t)}{L(t)}=1\quad \text{for every }\xi>0.$$  

\item[{\rm (b)}] The function 
$r\longmapsto L(r)$ is {\em slowly varying at (the right of) zero} if $t\longmapsto L(1/t)$ is slowly varying at
$\infty$.  

\item[{\rm (c)}]
A function $f$ is \emph{regularly varying at} $\infty$ (respectively, $0$) with real index $m$, in short $f\in RV_m(\infty)$ (respectively, 
$f\in RV_m(0+)$) if
$f(t)/t^m$ is slowly varying at $\infty$ (respectively, $0$). 
\end{enumerate}
\end{e-definition}

\begin{example} \label{example1}
{\rm Any positive constant function is trivially slowly varying at $\infty$. Other non-trivial examples of slowly varying functions
at $\infty$ are given by:
\begin{enumerate}
\item[(a)] The logarithm $\ln t$, its iterates $\ln_{n} t$
(defined as $\ln \ln_{n-1} t$) and powers of
$\ln_n t$ for any integer $n\geq 1$.
\item[(b)] $\exp\left(\frac{\ln t}{\ln \ln t}\right)$.
\item[(c)] $\exp((\ln t)^\nu)$ with $\nu\in (0,1)$.
\item[(d)] $ \exp\{(\ln t)^{1/3}\cos((\ln t)^{1/3})\}$. 
\end{enumerate}}
\end{example}

\begin{remark} \label{rlimo} {\rm Note that $\lim_{t\to \infty} f(t)=\infty $ 
(respectively, $0$) for any function $f\in RV_m(\infty)$ with $m>0$ 
(respectively, $m<0$). However, the limit at $\infty$ of a slowly varying 
function $L$ at $\infty$ cannot be determined in general, and it may not even exist (see example (d) above for which $\liminf_{t\to \infty}L(t)=0$ and
$\limsup_{t\to \infty}L(t)=\infty$). }
\end{remark}

\begin{proposition}[Uniform Convergence Theorem]
\label{pn} If $L$ is a slowly varying function at zero, then $L(\xi
t)/L(t)\to 1$ as $t\to 0$, uniformly on each compact
$\xi$-set in $(0,\infty)$.\end{proposition}

\begin{theorem}[Representation Theorem] \label{rep-th}
The function $L$ is slowly varying at $0$ if and only if we have 
$$ L(t)=\eta(t)\exp\left(\int_t^c\frac{\varepsilon(r)}{r}dr\right), \qquad 0<t\le c
$$
for some $c>0$, where $\eta$ is a measurable function on $(0,c]$ satisfying $\lim_{t\to 0^+}\eta(t)=\eta\in (0,\infty)$ and $\varepsilon$ is a continuous function on $(0,c]$ such that $\lim_{t\to 0^+}\varepsilon (t)=0$. 
\end{theorem}

\begin{remark} \label{nsw} If $\eta(t)$ is replaced by a positive constant $\eta$, then the new function $\eta$ is referred 
to as a {\em normalised slowly varying function}. In this case, 
$\varepsilon (t)=-t L'(t)/L(t)$ for $0<t\le c$. Conversely, any function $\tilde L \in C^1(0,c]$, which is positive and satisfies 
$\lim_{t \to 0^+} t {\tilde L}'(t)/{\tilde L}(t)=0$, is a normalised slowly varying function. 
\end{remark}

\begin{remark} \label{lol}
Any slowly varying function at zero is asymptotically equivalent to a normalised slowly varying one. 
\end{remark}

\begin{theorem}[Karamata's Theorem at $0$] \label{kar2}
Let $f$ vary regularly at zero with index $\rho$ and be locally bounded on $(0,c]$. The following assertions hold:
\begin{enumerate}
\item[{\rm (a)}] For any $j \leq -(\rho+1)$, we have
\[
\lim_{t\to 0^+} \frac{t^{j+1}f(t)}{  \int_t^c r^{j} f(r)\, dr} =-\left(j+\rho+1\right);
\]
\item[{\rm (b)}] For any $j>-(\rho+1)$ (and for $j=-(\rho+1)$ if $\int_{0^+} r^{-\rho-1}f(r)\, dr<+\infty$), we have
\[
\lim_{t\to 0^+} \frac{t^{j+1}f(t)}{\int_0^t r^{j}f(r)\,dr}= j+\rho+1.
\]
\end{enumerate} 
\end{theorem}

\begin{proposition}[Karamata's Theorem at $\infty$]
\label{p3} If $f\in RV_\rho(\infty)$ is locally bounded in $[A,\infty)$, then
\begin{enumerate}
\item[{\rm (a)}] 
For any $j\geq -(\rho+1)$, we have
$$ \lim_{t\to \infty}\frac{t^{j+1}f(t)}{\int_A^t \xi^j f(\xi)\,d\xi}=j+\rho+1.$$ 
\item[{\rm (b)}] For any $j<-(\rho+1)$ (and for $j=-(\rho+1)$ if
$\int^\infty \xi^{-(\rho+1)}f(\xi)\,d\xi<\infty$), we have
\[ \lim_{t\to \infty}\frac{t^{j+1}f(t)}{\int_t^\infty \xi^j f(\xi)\,d\xi}=
-(j+\rho+1).
\]
\end{enumerate}
\end{proposition}

As in \cite{resn}, we denote by $f^\leftarrow$ the (left continuous)
inverse of a non-decreasing function $f$ on $\RR$, namely
\[ f^\leftarrow (t)=\inf\{s:\ \ f(s)\geq t \}.\]

\begin{proposition}[see Proposition 0.8 in \cite{resn}] \label{p08} We have
\begin{enumerate}
\item If $f\in RV_\rho(\infty)$, then $\lim_{t\to \infty}\ln f(t)/\ln
t=\rho$. \item If $f_1\in RV_{\rho_1}(\infty)$ and $f_2\in RV_{\rho_2}(\infty)$
with $\lim_{t\to \infty}f_2(t)=\infty$, then \[ f_1\circ f_2\in
RV_{\rho_1 \rho_2}.\]
\item Suppose $f$ is non-decreasing,
$f(\infty)=\infty$, and $f\in RV_\rho(\infty)$ with $0<\rho<\infty$. Then
\[ f^\leftarrow \in RV_{1/\rho}(\infty).\]
\end{enumerate}
\end{proposition}

\begin{remark} \label{asb}
{\rm If $({\mathbf A_1})$--$({\mathbf A_3})$ hold, then  
by \cite[Lemma~A.7]{cd},  
there exist continuous functions $h_1$ and $h_2$ on $[0,\infty)$, positive on $(0,\infty)$ with $h_1(0)=h_2(0)=0$ such that 
\begin{equation} \label{9.4}
\begin{aligned}
 \begin{cases}
  h_1(t) \leq h(t) \leq h_2(t) \quad \text{for } t \in [0,\infty), \\
  h_1(t)/t^{p-1} \text{ and } h_2(t)/t^{p-1} \text{ are both increasing for } t \in (0,\infty),\\
  h_1(t)\sim h_2(t)\sim h(t) \ \ \text{as } t\to \infty.
 \end{cases}
\end{aligned}
\end{equation}
Therefore, without loss of generality, we can assume that $t\longmapsto t^{q-p+1} L_h(t)$ is increasing on $(0,\infty)$ so that $t^{q} L_h(t)$ is non-decreasing on
$(0,\infty)$. Moreover, as in \cite[Section 1.2.4]{florica}, we can take $L_h\in C^2[t_0,\infty)$ and $L_b\in C^2(0,r_0]$
for some large constant $t_0>0$ and $r_0\in (0,1)$ such that 
\neweq{nor} \lim_{t\to \infty} \frac{t L_h'(t)}{L_h(t)}=\lim_{t\to \infty} \frac{t^2 L_h''(t)}{L_h(t)}=0,\quad
\lim_{r\to 0^+} \frac{r L_b'(r)}{L_b(r)}=\lim_{r\to 0^+} \frac{r^2 L_b''(r)}{L_b(r)}=0.
\endeq }
\end{remark}

\end{document}